\documentclass[a4paper,10pt]{article}

\usepackage{amsfonts}
\usepackage{amsthm}
\usepackage{amsmath}
\usepackage{amssymb}
\usepackage{mathtools} 
\usepackage{enumerate}

\usepackage[top=1in, bottom=1in, left=1in, right=1in]{geometry}
\usepackage{xcolor}

\usepackage{graphicx}
\usepackage{epstopdf}
\usepackage{subcaption}
\usepackage{authblk}

\theoremstyle{plain}
\newtheorem{theorem}{Theorem}[section]
\newtheorem{corollary}[theorem]{Corollary}
\newtheorem{lemma}[theorem]{Lemma}

\theoremstyle{definition}
\newtheorem{definition}[theorem]{Definition}

\theoremstyle{remark}
\newtheorem{remark}[theorem]{Remark}

\begin{document}


\title{A naturally emerging bivariate Mittag-Leffler function and associated fractional-calculus operators}

\date{}

\author{Arran Fernandez\thanks{Email: \texttt{arran.fernandez@emu.edu.tr}}}
\author{Cemaliye K\"urt\thanks{Email: \texttt{cemaliye.kurt@emu.edu.tr}}}
\author{Mehmet Ali \"Ozarslan\thanks{Email: \texttt{mehmetali.ozarslan@emu.edu.tr}}}

\affil{{\small Department of Mathematics, Faculty of Arts and Sciences, Eastern Mediterranean University, North Cyprus, via Mersin-10, Turkey}}

\maketitle

\begin{abstract}
We define an analogue of the classical Mittag-Leffler function which is applied to two variables, and establish its basic properties. Using a corresponding single-variable function with fractional powers, we define an associated fractional integral operator which has many interesting properties. The motivation for these definitions is twofold: firstly their link with some fundamental fractional differential equations involving two independent fractional orders, and secondly the fact that they emerge naturally from certain applications in bioengineering. \\

\end{abstract}

\section{Introduction}

The classical \textbf{Mittag-Leffler function}, defined as
\begin{equation}
\label{ML:orig}
E_{\alpha}(x)=\sum_{n=0}^{\infty}\frac{x^n}{\Gamma(n\alpha+1)},\quad\quad\mathrm{Re}(\alpha)>0,
\end{equation}
was proposed \cite{mittag-leffler} by the Swedish mathematician G\"osta Mittag-Leffler in 1903. It has been extended and generalised in various ways \cite{gorenflo-kilbas-mainardi-rogosin,kilbas-srivastava-trujillo}, with functions denoted by $E_{\alpha,\beta}(x)$ and $E_{\alpha,\beta}^{\rho}(x)$ and $E_{\alpha,\beta,\gamma}(x)$, the ``two-parameter'' and ``three-parameter'' Mittag-Leffler functions, being defined by power series similar to the one for $E_{\alpha}(x)$ with modified coefficients to take account of the extra parameters:
\begin{align*}
E_{\alpha,\beta}(x)&=\sum_{n=0}^{\infty}\frac{x^n}{\Gamma(n\alpha+\beta)},\quad\quad\mathrm{Re}(\alpha)>0; \\
E_{\alpha,\beta}^{\rho}(x)&=\sum_{n=0}^{\infty}\frac{(\rho)_nx^n}{n!\Gamma(n\alpha+\beta)},\quad\quad\mathrm{Re}(\alpha)>0.
\end{align*}

Recently, a different type of generalisation has been proposed: the so-called ``bivariate'' and ``multivariate'' Mittag-Leffler functions, which are defined not by a power series in a single variable $x$, but by double power series in two variables $x$ and $y$, or even multiple power series in an arbitrary number of variables. Multivariate Mittag-Leffler functions were defined in \cite{saxena-kalla-saxena}, and various types of bivariate Mittag-Leffler functions have been defined in for example \cite{garg-manohar-kalla,ozarslan-kurt}. It is interesting to note that more than one type of bivariate function is emerging as an equivalent of the Mittag-Leffler function: so far we have seen
\begin{equation}
\label{ML:Garg}
E\left(\begin{array}{c}
\delta_1,\kappa_1;\delta_2,\kappa_2 \\ \gamma_1,\alpha_1,\beta_1;\gamma_2,\alpha_2;\gamma_3,\beta_2
\end{array}\Big|x,y\right)
=\sum_{m=0}^{\infty}\sum_{n=0}^{\infty}\frac{(\delta_1)_{\kappa_1m}(\delta_2)_{\kappa_2n}}{\Gamma(\alpha_1m+\beta_1n+\gamma_1)}\frac{x^m}{\Gamma(\alpha_2m+\gamma_2)}\frac{y^n}{\Gamma(\beta_2n+\gamma_3)}
\end{equation}
defined in \cite{garg-manohar-kalla} under the conditions $\mathrm{Re}(\alpha_j)>0$, $\mathrm{Re}(\beta_j)>0$, $\mathrm{Re}(\kappa_j)>0$, and
\begin{equation}
\label{ML:OK}
E_{\alpha ,\beta ,\kappa }^{(\delta )}(x,y)=\sum_{m=0}^{\infty}\sum_{n=0}^{\infty }\frac{(\delta )_{m+n}}{\Gamma (\alpha +m)\Gamma(\beta+\kappa n)}\frac{x^{m}}{m!}\frac{y^{\kappa n}}{n!},\quad\quad\alpha,\beta,\gamma\in\mathbb{C},
\end{equation}
defined in \cite{ozarslan-kurt} under the conditions $\mathrm{Re}(\alpha)>0$, $\mathrm{Re}(\beta)>0$, $\mathrm{Re}(\kappa)>0$.

These are both functions of two variables $x$ and $y$, both expressed as double power series which may be seen as analogous to the power series \eqref{ML:orig} for the original Mittag-Leffler function, but they are quite different functions from each other. The difference in structure is partly due to the fractional power of $y$ in \eqref{ML:OK}, but more importantly due to the exact nature of the gamma functions involved in the summand. In particular, in \eqref{ML:Garg} the only one involving both $m$ and $n$ (ensuring the double sum is not separable into a product of two single sums) is the $\Gamma(\alpha_1m+\beta_1n+\gamma_1)$ on the bottom, while in \eqref{ML:OK} the only one involving both $m$ and $n$ is the $(\delta)_{m+n}$ on the top.

In the current work, we shall define another type of bivariate Mittag-Leffler function, different from both \eqref{ML:Garg} and \eqref{ML:OK}, and then proceed to study its properties and applications. It is therefore necessary to justify why yet another bivariate Mittag-Leffler function is required, and what makes this new one specifically interesting. In order to do so, we turn to another field of mathematics which is strongly related to Mittag-Leffler functions.

\textbf{Fractional calculus} studies the generalisation of differentiation and integration to non-integer orders. This idea is centuries old, and classically \cite{miller-ross,oldham-spanier,samko-kilbas-marichev} it has been mostly tied to the Riemann--Liouville fractional calculus, defined by the following formulae for fractional integrals and fractional derivatives respectively:
\begin{align}
\label{RLorig:defint} \prescript{RL}{c}I_{x}^{\alpha}f(x)&=\frac{1}{\Gamma(\alpha)}\int_c^x(x-\xi)^{\alpha-1}f(\xi)\,\mathrm{d}\xi,\quad\quad\mathrm{Re}(\alpha)>0; \\
\label{RLorig:defderiv} \prescript{RL}{c}D_{x}^{\alpha}f(x)&=\frac{\mathrm{d}^k}{\mathrm{d}x^k}\prescript{RL}{c}I_{x}^{k-\alpha}f(x),k=\lfloor\mathrm{Re}(\alpha)\rfloor+1,\quad\quad\mathrm{Re}(\alpha)\geq0.
\end{align}
Closely related to the Riemann--Liouville fractional derivative, although not equivalent, is the Caputo fractional derivative. This is defined as follows using the Riemann--Liouville fractional integral:
\[
\prescript{C}{c}D_{x}^{\alpha}f(x)=\prescript{RL}{c}I_{x}^{k-\alpha}\frac{\mathrm{d}^k}{\mathrm{d}x^k}f(x),k=\lfloor\mathrm{Re}(\alpha)\rfloor+1,\quad\quad\mathrm{Re}(\alpha)\geq0.
\]
There are many other possible ways of defining fractional derivatives and integrals, including many which have been proposed only in the last ten years. This has led to various suggestions of criteria for what makes an operator a ``fractional derivative'' \cite{teodoro-machado-oliveira,hilfer-luchko}, as well as some proposed broad classes of fractional operators to cover many different definitions \cite{fernandez-ozarslan-baleanu,baleanu-fernandez2}.

There is a deep connection between Mittag-Leffler functions and fractional calculus, which has been explored in several texts \cite{gorenflo-kilbas-mainardi-rogosin,mathai-haubold}. One way in which this connection emerges is the following fractional differential relationship:
\begin{equation}
\label{ML:DE:Caputo}
\prescript{C}{0}D_{x}^{\alpha}\Big(E_{\alpha}(\omega x^{\alpha})\Big)=\omega E_{\alpha}(\omega x^{\alpha}),\quad\quad\mathrm{Re}(\alpha)>0.
\end{equation}
For this reason, the Mittag-Leffler function may be seen as a fractional equivalent of the exponential function. It and other power series related to \eqref{ML:orig} frequently emerge in the solution of fractional differential equations, by methods such as series solutions and numerical approximations \cite{sakamoto-yamamoto,momani-odibat,wang,djida-fernandez-area}.

Furthermore, some types of fractional calculus involve Mittag-Leffler functions inherently in their very definitions. These include the Prabhakar definition \cite{prabhakar,kilbas-saigo-saxena} which uses an integral transform with a 3-parameter Mittag-Leffler function in the kernel, and the Atangana--Baleanu definition \cite{atangana-baleanu,baleanu-fernandez} which uses an integral transform with a 1-parameter Mittag-Leffler function in the kernel. The latter especially has discovered many applications, since various real-life processes have behaviour which is better described by a Mittag-Leffler law than a power law \cite{baleanu-jajarmi-sajjadi-mozyrska,kumar-singh-tanwar-baleanu,yusuf-inc-aliyu-baleanu}. Recently, the newly proposed bivariate Mittag-Leffler functions have also been used to define fractional-calculus operators \cite{ozarslan-kurt,kurt-ozarslan-fernandez}.

Our current work is motivated by some ongoing experimental studies conducted at the Universities of Cambridge and London \cite{bonfanti-etal}, in which certain operators involving bivariate Mittag-Leffler functions emerged naturally from the analysis of the mechanical response of epithelial tissues. We undertook the task of properly defining these functions and operators, and proving their fundamental properties which may then be used later in analysis of those experimental results. We also discovered connections between the same bivariate Mittag-Leffler functions and certain important fractional differential equations. Therefore, we are confident that the functions and operators defined herein will have a rapid impact in several different fields of study. For a more detailed discussion of the impact and applications, see Section \ref{Sec:end} below.

This paper is arranged as follows. In section \ref{Sec:fn}, we define a new bivariate Mittag-Leffler function, both as a function of two independent variables $x,y$ and as a function of $t^{\alpha},t^{\beta}$ (the latter is technically univariate, but it still functions in a bivariate way, as we shall see). We consider various properties of these functions, including integral representations and relationships, Laplace transforms, fractional integrals and derivatives, and relationships with bivariate Laguerre polynomials. In section \ref{Sec:op}, we use the function applied to $t^{\alpha},t^{\beta}$ to define a new fractional integral operator and associated fractional differential operators, which we then analyse to establish their fundamental properties. In section \ref{Sec:end}, we discuss the conclusions, applications, and future impact of our work here.

\section{The new bivariate Mittag-Leffler function} \label{Sec:fn}

\subsection{Establishing the definition}

The following is the main definition to start off our work here.

\begin{definition}
\label{Def:ourML}
Let $\alpha ,\beta ,\gamma ,\delta \in \mathbb{C}$ be complex parameters with $\mathrm{Re}(\alpha )>0$ and $\mathrm{Re}(\beta )>0$. We define the following bivariate Mittag-Leffler function for general complex numbers $x$ and $y$:
\begin{equation}
\label{3}
E_{\alpha ,\beta ,\gamma }^{\delta }(x,y) :=\sum_{k=0}^{\infty}\sum_{l=0}^{\infty }\frac{(\delta )_{k+l}}{\Gamma (\alpha k+\beta l+\gamma )}\cdot\frac{x^ky^l}{k!l!},
\end{equation}
where the numerator $(\delta)_{k+l}$ is the Pochhammer symbol defined by \[(a)_n=\frac{\Gamma(a+n)}{\Gamma(a)}=a(a+1)(a+2)\dots(a+n-1).\]
Note that the series in \eqref{3} converges absolutely and locally
uniformly, therefore defines an entire function in each of $x,y$, provided that $\mathrm{Re}(\alpha )>0$ and $\mathrm{Re}(\beta )>0$. This can be proved by using the criteria of Srivastava and Daoust \cite{srivastava-daoust} for the generalized Lauricella series in two variables; it is also clear intuitively that $\alpha$ and $\beta$, being multiplied by $k$ and $l$ inside the gamma function, should be in the ``right direction'' while the values of $\gamma$ and $\delta$ will not matter for convergence.
\end{definition}

\begin{remark}
It is necessary to compare Definition \ref{Def:ourML} with the alternative multivariate Mittag-Leffler function $E_{(\rho_1,\dots,\rho_m),\lambda}^{(\gamma_1,\dots,\gamma_m)}(z_1,\dots,z_m)$ defined by Saxena et al \cite{saxena-kalla-saxena}, specifically the bivariate version of their definition:
\begin{equation}
\label{ML:SKS}
E_{(\rho_1,\rho_2),\lambda}^{(\gamma_1,\gamma_2)}(x,y)=\sum_{k=0}^{\infty}\sum_{l=0}^{\infty}\frac{(\gamma_1)_k(\gamma_2)_l}{\Gamma(\rho_1k+\rho_2l+\lambda)}\cdot\frac{x^ky^l}{k!l!}.
\end{equation}
The only difference between their \eqref{ML:SKS} and our \eqref{3} is in the Pochhammer symbols on the numerator. The definition \eqref{ML:SKS} contains two separate Pochhammer symbols in $k$ and $l$, so the only thing making the double sum inseparable is the gamma function on the denominator. Our definition contains one Pochhammer symbol $(\delta)_{k+l}$, introducing another element of inseparability. This turns out to be crucial for the semigroup property which we prove in Theorem \ref{Thm:semigroup} below, in comparison with the definition of Saxena et al which does not possess a semigroup property.
\end{remark}

Of particular importance for our results and applications is the case where we write $x=\omega_1t^{\alpha}$ and $y=\omega_2t^{\beta}$ for a single variable $t$, and (optionally) multiply by an extra power function:
\begin{equation}
\label{5}
t^{\gamma-1}E_{\alpha ,\beta ,\gamma }^{\delta }(\omega_1t^{\alpha },\omega_2t^{\beta})=\sum_{k=0}^{\infty }\sum_{l=0}^{\infty }\frac{(\delta )_{k+l}}{\Gamma(\alpha k+\beta l+\gamma )}\cdot\frac{\omega_1^k\omega_2^l}{k!l!}\;t^{\alpha k+\beta l+\gamma-1}.
\end{equation}
This function is technically univariate, since it depends on a single variable $t$ in addition to the parameters $\alpha,\beta,\gamma,\delta$. But in many ways it functions similarly to a bivariate Mittag-Leffler function: it is still defined by a double series instead of a single one, and the separate parameters $\alpha$ and $\beta$ grant some independence to the two inputs $\omega_1t^{\alpha}$ and $\omega_2t^{\beta}$.

We shall see later that many important results can only be proved for the univariate function \eqref{5} rather than with fully independent variables $x,y$ as in \eqref{3}. This is because the parameters $\alpha$ and $\beta$ both appear together in $k\alpha+l\beta$ within the gamma function on the denominator, so it is often necessary to have $k\alpha+l\beta$ in the exponent of the power function on the numerator too. In \eqref{5}, the power function exponent and the argument of the gamma function are matching just as they should for (e.g.) Laplace transforms or fractional derivatives.

To provide motivation for our definition, we prove a few fundamental properties of the newly defined function $E_{\alpha,\beta,\gamma}^{\delta}(x,y)$ which are analogous to important properties of the original Mittag-Leffler function $E_{\alpha}(x)$.

\begin{lemma}
\label{Lem:1exp}
If all parameters are $1$, we recover the double exponential function:
\[E_{1,1,1}^1(x,y)=e^xe^y,\quad\quad x,y\in\mathbb{C}.\]
\end{lemma}

\begin{proof}
From the definition \eqref{3}:
\begin{align*}
E_{1,1,1}^{1}(x,y)&=\sum_{k=0}^{\infty}\sum_{l=0}^{\infty }\frac{(1)_{k+l}}{\Gamma(k+l+1)}\cdot\frac{x^ky^l}{k!l!}=\sum_{k=0}^{\infty}\sum_{l=0}^{\infty }\frac{\Gamma(k+l+1)}{\Gamma(1)\Gamma(k+l+1)}\cdot\frac{x^ky^l}{k!l!}=\sum_{k=0}^{\infty}\sum_{l=0}^{\infty }\frac{x^k}{k!}\cdot\frac{y^l}{l!}=e^ke^l,
\end{align*}
where we see that the parts of the expression that combine $k$ and $l$ together, namely $(\delta)_{k+l}$ and $\Gamma(\alpha k+\beta l+1)$, cancel out so that the two sums are completely separable.
\end{proof}

Lemma \ref{Lem:1exp} is the natural analogue of the fact that the original Mittag-Leffler function reduces to exponential when $\alpha=1$:
\[E_1(x)=e^x.\]
It is a fundamental aspect of the Mittag-Leffler function that it can be seen as a ``fractional exponential function'', the series \eqref{ML:orig} being like the Taylor series of $e^x$ but with an extra parametrisation given by $\alpha$.

To illustrate the functions that we have defined and continue to analyse in this paper, we include figures showing their graphs for some example values of the parameters involved. These graphs were generated using Mathematica, version 11.2, and using finite double sums ($0\leq k,l\leq20$) to approximate the infinite double series over $k,l$.
\begin{itemize}
\item Figure \ref{Fig:bivar} shows the bivariate function \eqref{3} plotted against $x$ and $y$, with the $x$-axis from left to right and the $y$-axis from front to back. In these graphs we assume $\gamma=\delta=1$ and use varying values of $\alpha$ and $\beta$ to capture different behaviours. \\

The first graph (Figure \ref{Fig:bivar:1,1}) shows the case of Lemma \ref{Lem:1exp}, namely the exponential function $e^{x+y}$. In the other graphs, we make small changes to the values of $\alpha$ and $\beta$ in turn and observe the changes to the graphs. It is noticeable that the function grows respectively faster or slower with respect to $x$ when $\alpha$ decreases or increases, and similarly for $y$ and $\beta$. This observation is borne out by the definition \eqref{3}: increasing $\alpha$ would increase the growth of the denominator with respect to $k$, which reduces the growth of the series with respect to $x$. \\

\item Figure \ref{Fig:univar} shows the univariate version \eqref{5} plotted against $t$, again assuming that $\gamma=\delta=1$ and allowing $\alpha,\beta$ to vary. This function is symmetric in $\alpha$ and $\beta$, so instead of changing each parameter in turn, we allow both to vary together. \\

The first graph (Figure \ref{Fig:univar:1,1}) shows the case of Lemma \ref{Lem:1exp}, namely the exponential function $e^{2t}$. In the other graphs, we increase or decrease the parameters $\alpha,\beta$ and observe the changes to the graphs. Again, the growth of the function (as apparent from the vertical axis) is respectively faster or slower if $\alpha$ and $\beta$ are smaller or larger. This makes sense because increasing the parameters means increasing the growth of the denominator in \eqref{5}, which means the series would grow more slowly.
\end{itemize}

\begin{figure}	
\centering
	\begin{subfigure}[t]{0.4\textwidth}
		\centering
		\includegraphics[width=\textwidth]{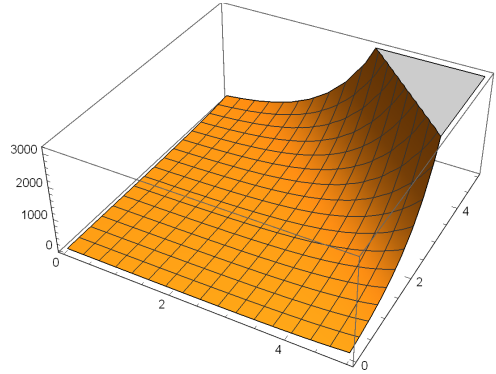}
		\caption{$\alpha=1,\beta=1$: here $E_{1,1,1}^1(x,y)=e^{x+y}$}\label{Fig:bivar:1,1}
	\end{subfigure}
\\
	\begin{subfigure}{0.4\textwidth}
		\centering
		\includegraphics[width=\textwidth]{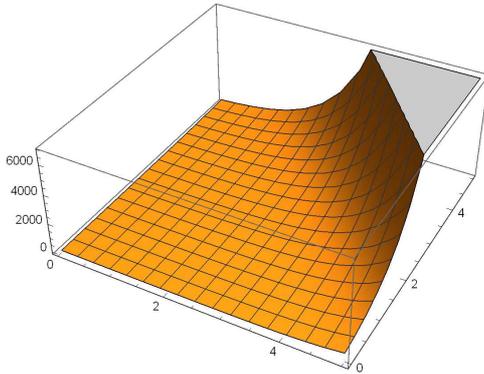}
		\caption{$\alpha=0.9,\beta=1$}\label{Fig:bivar:0.9,1}
	\end{subfigure}
\quad
	\begin{subfigure}{0.4\textwidth}
		\centering
		\includegraphics[width=\textwidth]{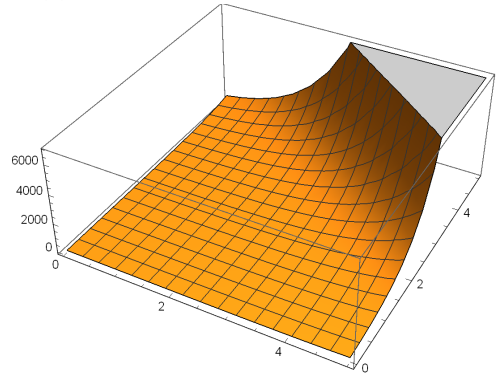}
		\caption{$\alpha=1,\beta=0.9$}\label{Fig:bivar:1,0.9}
	\end{subfigure}
\\
		\begin{subfigure}[b]{0.4\textwidth}
		\centering
		\includegraphics[width=\textwidth]{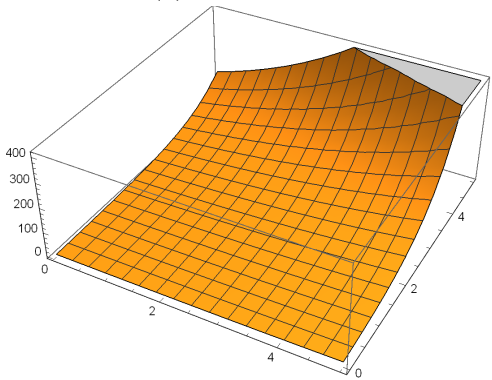}
		\caption{$\alpha=1.5,\beta=1$}\label{Fig:bivar:1.5,1}
	\end{subfigure}
\quad
	\begin{subfigure}[b]{0.4\textwidth}
		\centering
		\includegraphics[width=\textwidth]{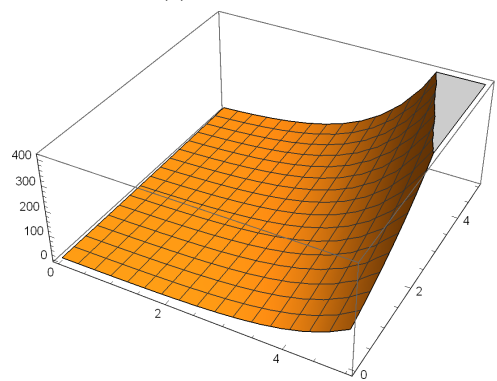}
		\caption{$\alpha=1,\beta=1.5$}\label{Fig:bivar:1,1.5}
	\end{subfigure}
	\caption{Plots of the bivariate Mittag-Leffler function $E_{\alpha,\beta,1}^{1}(x,y)$ with $\gamma=\delta=1$ and varying $\alpha,\beta$}\label{Fig:bivar}
\end{figure}

\begin{figure}	
\centering
	\begin{subfigure}[t]{0.4\textwidth}
		\centering
		\includegraphics[width=\textwidth]{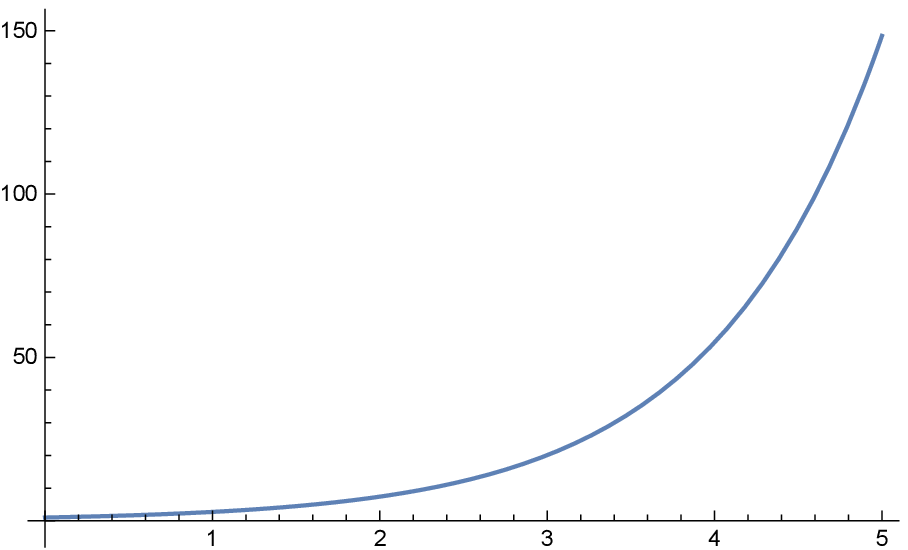}
		\caption{$\alpha=1,\beta=1$: here $E_{1,1,1}^1(t,t)=e^{2t}$}\label{Fig:univar:1,1}
	\end{subfigure}
\quad
	\begin{subfigure}[t]{0.4\textwidth}
		\centering
		\includegraphics[width=\textwidth]{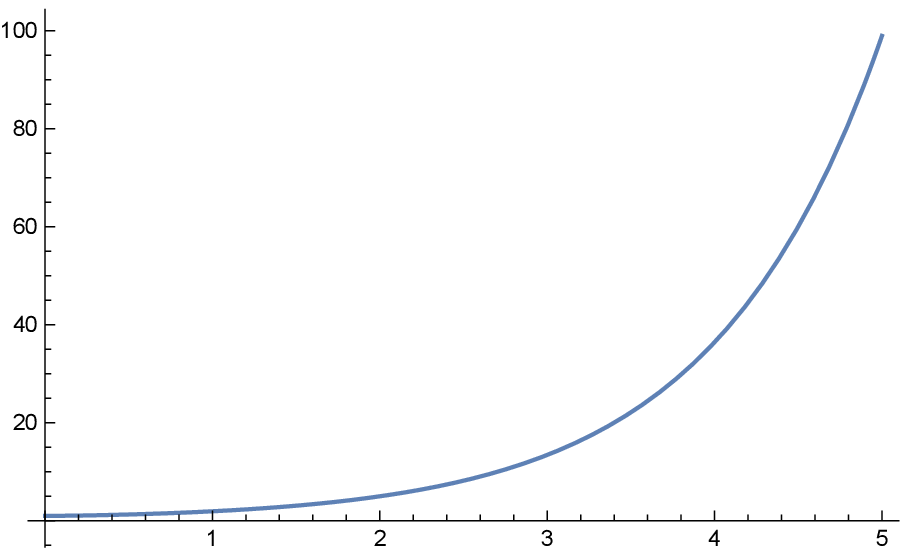}
		\caption{$\alpha=1.5,\beta=1.5$}\label{Fig:univar:1.5,1.5}
	\end{subfigure}
\\
		\begin{subfigure}[b]{0.4\textwidth}
		\centering
		\includegraphics[width=\textwidth]{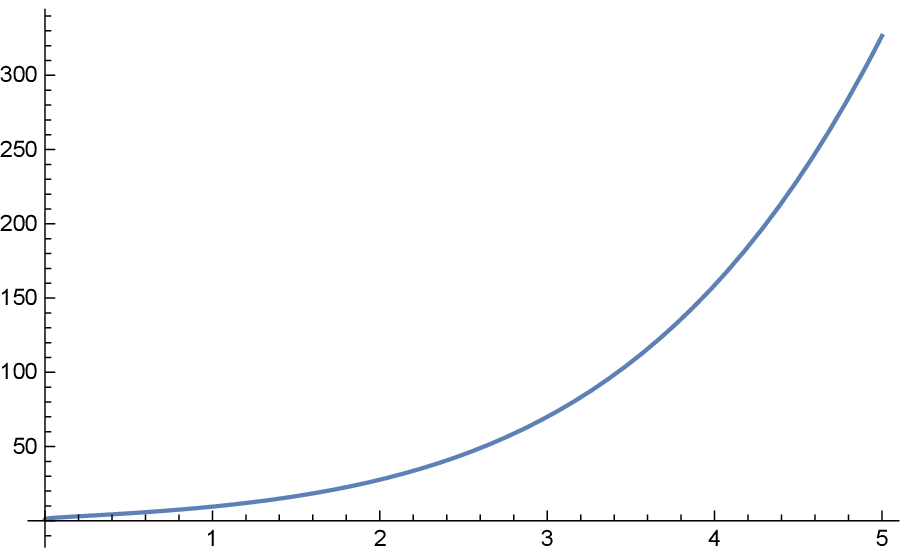}
		\caption{$\alpha=0.25,\beta=0.25$}\label{Fig:univar:0.25,0.25}
	\end{subfigure}
\quad
	\begin{subfigure}[b]{0.4\textwidth}
		\centering
		\includegraphics[width=\textwidth]{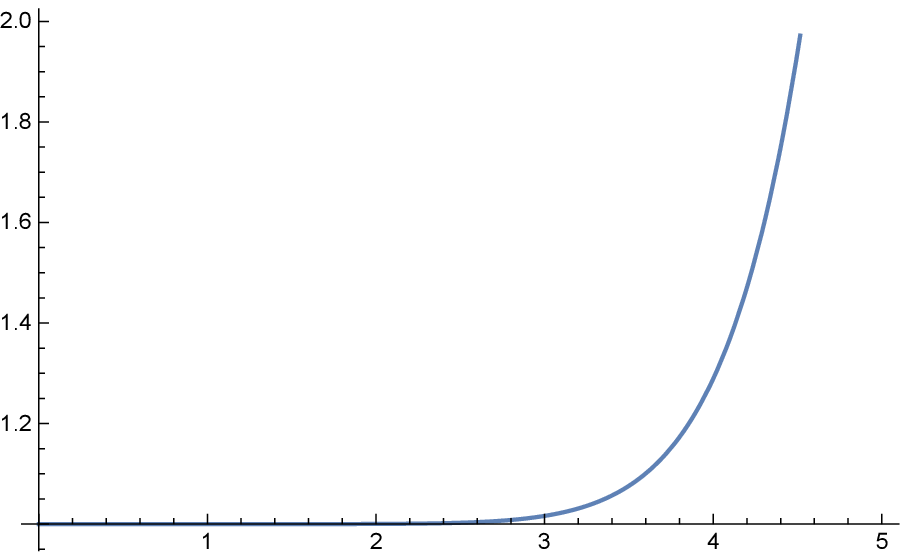}
		\caption{$\alpha=10,\beta=10$}\label{Fig:univar:10,10}
	\end{subfigure}
	\caption{Plots of the univariate version $E_{\alpha,\beta,1}^1(t^{\alpha},t^{\beta})$ with $\gamma=\delta=1$ and varying $\alpha,\beta$}\label{Fig:univar}
\end{figure}

\begin{theorem}
\label{Thm:FDE}
The univariate form \eqref{5} of the bivariate Mittag-Leffler function \eqref{3}, with $\delta=1$, gives a solution function
\[u(t)=t^{\gamma-1}E_{\alpha,\beta,\gamma}^{1}(\omega_1t^{\alpha},\omega_2t^{\beta})\]
for the following fractional initial value problem in the form of an ODE involving two independent fractional orders of differentiation:
\begin{equation}
\label{FDE:RL}
\prescript{RL}{0}D^{\alpha+\beta}_tu(t)-\omega_2\prescript{RL}{0}D^{\alpha}_tu(t)-\omega_1\prescript{RL}{0}D^{\beta}_tu(t)=\frac{t^{\gamma-\alpha-\beta-1}}{\Gamma(\gamma-\alpha-\beta)}.
\end{equation}
\end{theorem}

\begin{proof}
We shall make use of the well-known formula
\begin{equation}
\label{power:RL}
\prescript{RL}{c}D^{\nu}_x\left(\frac{x^{\mu}}{\Gamma(\mu+1)}\right)=\frac{x^{\mu-\nu}}{\Gamma(\mu-\nu+1)},\quad\quad\mu,\nu\in\mathbb{C},\mathrm{Re}(\mu)>-1,
\end{equation}
for Riemann--Liouville fractional differintegrals of power functions. Starting from the series formula \eqref{5}, we compute the fractional differintegrals of $u(t)$ as follows:
\begin{align*}
\prescript{RL}{0}D^{\alpha}_tu(t)&=\prescript{RL}{0}D^{\alpha}_t\left(\sum_{k=0}^{\infty }\sum_{l=0}^{\infty }\frac{(k+l)!}{\Gamma(\alpha k+\beta l+\gamma )}\cdot\frac{\omega_1^k\omega_2^l}{k!l!}\;t^{\alpha k+\beta l+\gamma-1}\right) \\
&=\sum_{k=0}^{\infty }\sum_{l=0}^{\infty }\frac{(k+l)!\omega_1^k\omega_2^l}{k!l!}\;\prescript{RL}{0}D^{\alpha}_t\left(\frac{t^{\alpha k+\beta l+\gamma-1}}{\Gamma(\alpha k+\beta l+\gamma )}\right) \\
&=\sum_{k=0}^{\infty }\sum_{l=0}^{\infty }\frac{(k+l)!\omega_1^k\omega_2^l}{k!l!}\;\frac{t^{\alpha(k-1)+\beta l+\gamma-1}}{\Gamma(\alpha(k-1)+\beta l+\gamma )} \\
&=\sum_{l=0}^{\infty }\frac{\omega_2^lt^{-\alpha+\beta l+\gamma-1}}{\Gamma(-\alpha+\beta l+\gamma)}+\sum_{k=1}^{\infty }\sum_{l=0}^{\infty }\frac{(k+l)!\omega_1^k\omega_2^l}{k!l!}\;\frac{t^{\alpha(k-1)+\beta l+\gamma-1}}{\Gamma(\alpha(k-1)+\beta l+\gamma )} \\
&=t^{\gamma-\alpha-1}E_{\beta,\gamma-\alpha}(\omega_2t^{\beta})+\sum_{k=0}^{\infty }\sum_{l=0}^{\infty }\frac{(k+l+1)!\omega_1^{k+1}\omega_2^l}{(k+1)!l!}\;\frac{t^{\alpha k+\beta l+\gamma-1}}{\Gamma(\alpha k+\beta l+\gamma )},
\end{align*}
and similarly (by symmetry)
\begin{align*}
\prescript{RL}{0}D^{\beta}_tu(t)&=t^{\gamma-\beta-1}E_{\alpha,\gamma-\beta}(\omega_1t^{\alpha})+\sum_{k=0}^{\infty }\sum_{l=0}^{\infty }\frac{(k+l+1)!\omega_1^k\omega_2^{l+1}}{k!(l+1)!}\;\frac{t^{\alpha k+\beta l+\gamma-1}}{\Gamma(\alpha k+\beta l+\gamma )}.
\end{align*}
Taking a linear combination, we find
\begin{align*}
&\omega_2\prescript{RL}{0}D^{\alpha}_tu(t)+\omega_1\prescript{RL}{0}D^{\beta}_tu(t) \\
&\hspace{1cm}=\omega_2t^{\gamma-\alpha-1}E_{\beta,\gamma-\alpha}(\omega_2t^{\beta})+\omega_1t^{\gamma-\beta-1}E_{\alpha,\gamma-\beta}(\omega_1t^{\alpha})+\omega_2\sum_{k=0}^{\infty }\sum_{l=0}^{\infty }\frac{(k+l+1)!\omega_1^{k+1}\omega_2^l}{(k+1)!l!}\;\frac{t^{\alpha k+\beta l+\gamma-1}}{\Gamma(\alpha k+\beta l+\gamma )} \\ &\hspace{3cm}+\omega_1\sum_{k=0}^{\infty }\sum_{l=0}^{\infty }\frac{(k+l+1)!\omega_1^k\omega_2^{l+1}}{k!(l+1)!}\;\frac{t^{\alpha k+\beta l+\gamma-1}}{\Gamma(\alpha k+\beta l+\gamma )} \\
&\hspace{1cm}=\omega_2t^{\gamma-\alpha-1}E_{\beta,\gamma-\alpha}(\omega_2t^{\beta})+\omega_1t^{\gamma-\beta-1}E_{\alpha,\gamma-\beta}(\omega_1t^{\alpha}) \\ &\hspace{3cm}+\omega_1\omega_2\sum_{k=0}^{\infty }\sum_{l=0}^{\infty }\left[\frac{(k+l+1)!}{(k+1)!l!}+\frac{(k+l+1)!}{k!(l+1)!}\right]\frac{\omega_1^k\omega_2^lt^{\alpha k+\beta l+\gamma-1}}{\Gamma(\alpha k+\beta l+\gamma )} \\
&\hspace{1cm}=\omega_2t^{\gamma-\alpha-1}E_{\beta,\gamma-\alpha}(\omega_2t^{\beta})+\omega_1t^{\gamma-\beta-1}E_{\alpha,\gamma-\beta}(\omega_1t^{\alpha})+\omega_1\omega_2\sum_{k=0}^{\infty }\sum_{l=0}^{\infty }\frac{(k+l+2)!}{(k+1)!(l+1)!}\frac{\omega_1^k\omega_2^lt^{\alpha k+\beta l+\gamma-1}}{\Gamma(\alpha k+\beta l+\gamma )}.
\end{align*}

Meanwhile, starting again with the series definition from \eqref{6} for $u(t)$, we find:
\begin{align*}
\prescript{RL}{0}D^{\alpha+\beta}_tu(t)&=\sum_{k=0}^{\infty }\sum_{l=0}^{\infty }\frac{(k+l)!\omega_1^k\omega_2^l}{k!l!}\;\prescript{RL}{0}D^{\alpha+\beta}_t\left(\frac{t^{\alpha k+\beta l+\gamma-1}}{\Gamma(\alpha k+\beta l+\gamma )}\right) \\
&=\sum_{k=0}^{\infty }\sum_{l=0}^{\infty }\frac{(k+l)!\omega_1^k\omega_2^l}{k!l!}\;\frac{t^{\alpha(k-1)+\beta(l-1)+\gamma-1}}{\Gamma(\alpha(k-1)+\beta(l-1)+\gamma )} \\
&=\left(\mathop{\sum\sum}_{k=0\ \text{or}\ l=0}+\sum_{k=1}^{\infty }\sum_{l=1}^{\infty }\right)\frac{(k+l)!\omega_1^k\omega_2^l}{k!l!}\;\frac{t^{\alpha(k-1)+\beta(l-1)+\gamma-1}}{\Gamma(\alpha(k-1)+\beta(l-1)+\gamma )} \\
&=\frac{t^{-\alpha-\beta+\gamma-1}}{\Gamma(-\alpha-\beta+\gamma)}+\sum_{k=1}^{\infty}\frac{\omega_1^kt^{\alpha(k-1)-\beta+\gamma-1}}{\Gamma(\alpha(k-1)-\beta+\gamma)}+\sum_{l=1}^{\infty}\frac{\omega_2^lt^{-\alpha+\beta(l-1)+\gamma-1}}{\Gamma(-\alpha+\beta(l-1)+\gamma )} \\ &\hspace{5cm}+\sum_{k=1}^{\infty }\sum_{l=1}^{\infty }\frac{(k+l)!\omega_1^k\omega_2^l}{k!l!}\;\frac{t^{\alpha(k-1)+\beta(l-1)+\gamma-1}}{\Gamma(\alpha(k-1)+\beta(l-1)+\gamma )} \\
&=\frac{t^{-\alpha-\beta+\gamma-1}}{\Gamma(-\alpha-\beta+\gamma)}+\sum_{k=0}^{\infty}\frac{\omega_1^{k+1}t^{\alpha k-\beta+\gamma-1}}{\Gamma(\alpha k-\beta+\gamma)}+\sum_{l=0}^{\infty}\frac{\omega_2^{l+1}t^{-\alpha+\beta l+\gamma-1}}{\Gamma(-\alpha+\beta l+\gamma)} \\ &\hspace{5cm}+\sum_{k=0}^{\infty }\sum_{l=0}^{\infty }\frac{(k+l+2)!\omega_1^{k+1}\omega_2^{l+1}}{(k+1)!(l+1)!}\;\frac{t^{\alpha k+\beta l+\gamma-1}}{\Gamma(\alpha k+\beta l+\gamma )} \\
&=\frac{t^{-\alpha-\beta+\gamma-1}}{\Gamma(-\alpha-\beta+\gamma)}+\omega_1t^{\gamma-\beta-1}E_{\alpha,\gamma-\beta}(\omega_1t^{\alpha})+\omega_2t^{\gamma-\alpha-1}E_{\beta,\gamma-\alpha}(\omega_2t^{\beta}) \\ &\hspace{5cm}+\sum_{k=0}^{\infty }\sum_{l=0}^{\infty }\frac{(k+l+2)!\omega_1^{k+1}\omega_2^{l+1}}{(k+1)!(l+1)!}\;\frac{t^{\alpha k+\beta l+\gamma-1}}{\Gamma(\alpha k+\beta l+\gamma )} \\
&=\frac{t^{-\alpha-\beta+\gamma-1}}{\Gamma(-\alpha-\beta+\gamma)}+\omega_2\prescript{RL}{0}D^{\alpha}_tu(t)+\omega_1\prescript{RL}{0}D^{\beta}_tu(t).
\end{align*}
Thus, the result is proved.
\end{proof}

Theorem \ref{Thm:FDE} is a natural analogue, in two independent fractional orders of differentiation $\alpha,\beta$, of the following differential equation which is satisfied by the original Mittag-Leffler function $E_{\alpha}(\omega x^{\alpha})$:
\begin{equation}
\label{ML:DE:RL}
\prescript{RL}{0}D_{x}^{\alpha}u(x)-\omega u(x)=\frac{x^{-\alpha}}{\Gamma(1-\alpha)},\quad\quad\mathrm{Re}(\alpha)>0.
\end{equation}
This fractional differential equation, using Riemann--Liouville derivatives, is more complicated than the corresponding Caputo equation satisfied by the same function $E_{\alpha}(\omega x^{\alpha})$, namely $\prescript{RL}{0}D_{x}^{\alpha}u(x)-\omega u(x)=0$ as stated in \eqref{ML:DE:Caputo} above. The difference arises from the constant term in the Mittag-Leffler function.

In our case, we have three different fractional derivatives to three different orders appearing in the equation \eqref{FDE:RL}. Changing Riemann--Liouville to Caputo will introduce three extra terms (from the constant term in the double Mittag-Leffler sum), so while the power function in \eqref{FDE:C} does disappear, it is replaced by two more power terms. Therefore, there is no fractional differential equation quite so nice as \eqref{ML:DE:Caputo} for our new Mittag-Leffler function. The Caputo differential equation is given in Corollary \ref{Coroll:FDE} below.

Note that, apart from the inhomogeneous forcing term, we do have a simple and elegant differential equation for the new Mittag-Leffler function \eqref{5}. This indicates its naturality as a bivariate (or double-series) version of the original Mittag-Leffler function.

\begin{corollary}
\label{Coroll:FDE}
The univariate form \eqref{5} of the bivariate Mittag-Leffler function \eqref{3}, with $\gamma=\delta=1$ and under the extra assumption $0<\mathrm{Re}(\alpha+\beta)<1$, gives a solution function
\[u(t)=E_{\alpha,\beta,1}^{1}(\omega_1t^{\alpha},\omega_2t^{\beta})\]
for the following fractional initial value problem in the form of an ODE involving two independent fractional orders of differentiation:
\begin{equation}
\label{FDE:C}
\prescript{C}{0}D^{\alpha+\beta}_tu(t)-\omega_2\prescript{C}{0}D^{\alpha}_tu(t)-\omega_1\prescript{C}{0}D^{\beta}_tu(t)=\frac{\omega_1t^{-\alpha}}{\Gamma(1-\alpha)}+\frac{\omega_2t^{-\beta}}{\Gamma(1-\beta)},\quad\quad u(0)=1.
\end{equation}
\end{corollary}

\begin{proof}
In the Caputo case, we have the same result \eqref{power:RL} as Riemann--Liouville for fractional derivatives of power functions, but only under the assumption that $\mathrm{Re}(\mu)>\lfloor\mathrm{Re}(\nu)\rfloor$. This is because the Caputo derivative is a Riemann--Liouville integral of a standard repeated derivative,
\[\prescript{C}{0}D^{\nu}_x\left(\frac{x^{\mu}}{\Gamma(\mu+1)}\right)=\prescript{RL}{0}I^{\nu-n}_x\left(\frac{x^{\mu-n}}{\Gamma(\mu-n+1)}\right),\quad n:=\lfloor\mathrm{Re}(\nu)\rfloor+1,\]
and so to get the final conclusion \eqref{power:RL} we need the assumption that $\mathrm{Re}(\mu-n)>-1$, namely that $\mathrm{Re}(\mu)>n-1=\lfloor\mathrm{Re}(\nu)\rfloor$. If this condition does not hold, then the Caputo derivative does not exist (the integral is divergent). Therefore, for the Caputo derivatives in \eqref{FDE:C} to be defined, we need the conditions $\mathrm{Re}(\alpha)>\lfloor\mathrm{Re}(\alpha+\beta)\rfloor$ and $\mathrm{Re}(\beta)>\lfloor\mathrm{Re}(\alpha+\beta)\rfloor$, which necessitates the assumption $0<\mathrm{Re}(\alpha+\beta)<1$.

Also, the Caputo derivative of a constant is always zero, but the Riemann--Liouville derivative is not: $\prescript{RL}{0}D^{\nu}_x(1)=\frac{x^{-\nu}}{\Gamma(1-\nu)}$. So when we take Caputo derivatives of the series for $u(t)$ in the case $\gamma=1$, namely
\[u(t)=\sum_{k=0}^{\infty }\sum_{l=0}^{\infty }\frac{(k+l)!}{\Gamma(\alpha k+\beta l+1)}\cdot\frac{\omega_1^k\omega_2^l}{k!l!}\;t^{\alpha k+\beta l},\]
the constant term ($k=l=0$) goes to zero instead of a negative power function as in the Riemann--Liouville case. This means each Caputo derivative differs from the corresponding Riemann--Liouville derivative by a negative power function:
\begin{align*}
\prescript{C}{0}D^{\alpha}_tu(t)=\prescript{RL}{0}D^{\alpha}_tu(t)-\frac{t^{-\alpha}}{\Gamma(1-\alpha)}&,\quad\quad\prescript{C}{0}D^{\beta}_tu(t)=\prescript{RL}{0}D^{\beta}_tu(t)-\frac{t^{-\beta}}{\Gamma(1-\beta)}, \\ \prescript{C}{0}D^{\alpha+\beta}_tu(t)&=\prescript{RL}{0}D^{\alpha+\beta}_tu(t)-\frac{t^{-\alpha-\beta}}{\Gamma(1-\alpha-\beta)}.
\end{align*}
From Theorem \ref{Thm:FDE} we know that
\[
\prescript{RL}{0}D^{\alpha+\beta}_tu(t)-\frac{t^{-\alpha-\beta}}{\Gamma(1-\alpha-\beta)}-\omega_2\prescript{RL}{0}D^{\alpha}_tu(t)-\omega_1\prescript{RL}{0}D^{\beta}_tu(t)=0,
\]
which implies
\[
\prescript{C}{0}D^{\alpha+\beta}_tu(t)-\omega_2\left[\prescript{C}{0}D^{\alpha}_tu(t)+\frac{t^{-\alpha}}{\Gamma(1-\alpha)}\right]-\omega_1\left[\prescript{C}{0}D^{\beta}_tu(t)+\frac{t^{-\beta}}{\Gamma(1-\beta)}\right]=0,
\]
and the result follows.
\end{proof}

\subsection{Results and relationships}

Having introduced the main functions that we shall be studying, and provided motivation for the naturality of these definitions as extensions of the classical Mittag-Leffler function, we proceed to prove various results about these functions. To begin with, the following complex integral representation is similar to the one proved for a different series function in \cite{ozarslan}.

\begin{theorem}[Complex integral representation]
\label{Thm:complint}
For $\alpha,\beta,\gamma,\delta\in\mathbb{C}$ with $\mathrm{Re}(\alpha)>0$ and $\mathrm{Re}(\beta)>0$, the bivariate Mittag-Leffler function \eqref{3} has the following complex integral representation:
\[
E_{\alpha,\beta,\gamma}^{\delta}(x,y)=\frac{1}{2\pi i}\int_H\frac{e^tt^{-\gamma}}{(1-xt^{-\alpha}-yt^{-\beta})^{\delta}}\,\mathrm{d}t,
\]
where $H$ is the Hankel contour in the complex plane, which starts and ends at $-\infty$ and wraps around the origin in a positive direction, and where all fractional powers of $t$ are defined using the principal branch with branch cut $(-\infty,0]$.
\end{theorem}

\begin{proof}
We make use of the Hankel formula for the gamma function \cite[\S12]{whittaker-watson}:
\[
\frac{1}{\Gamma(z)}=\frac{1}{2\pi i}\int_Ht^{-z}e^t\,\mathrm{d}t,\quad\quad z\in\mathbb{C}.
\]
Using this for the inverse gamma function in the series for $E_{\alpha,\beta,\gamma}^{\delta}(x,y)$:
\begin{align*}
E_{\alpha ,\beta ,\gamma }^{\delta }(x,y)&=\sum_{k=0}^{\infty}\sum_{l=0}^{\infty }\frac{(\delta )_{k+l}x^ky^l}{k!l!\Gamma (\alpha k+\beta l+\gamma )} \\
&=\sum_{k=0}^{\infty}\sum_{l=0}^{\infty }\frac{(\delta )_{k+l}x^ky^l}{2\pi ik!l!}\int_Ht^{-\alpha k-\beta l-\gamma}e^t\,\mathrm{d}t \\
&=\frac{1}{2\pi i}\int_He^tt^{-\gamma}\sum_{k=0}^{\infty}\sum_{l=0}^{\infty }\frac{(\delta )_{k+l}}{k!l!}\left(\frac{x}{t^{\alpha}}\right)^k\left(\frac{y}{t^{\beta}}\right)^l\,\mathrm{d}t,
\end{align*}
where we use the local uniform convergence of the series to swap the summation and integration.

The double sum inside this integral may be evaluated as follows:
\begin{align*}
\sum_{k=0}^{\infty}\sum_{l=0}^{\infty }\frac{(\delta+k)_l(\delta)_k}{k!l!}\left(\frac{x}{t^{\alpha}}\right)^k\left(\frac{y}{t^{\beta}}\right)^l&=\sum_{k=0}^{\infty}\frac{(\delta)_k}{k!}\left(\frac{x}{t^{\alpha}}\right)^k\left(1-\frac{y}{t^{\beta}}\right)^{-\delta-k} \\
&=\left(1-\frac{y}{t^{\beta}}\right)^{-\delta}\sum_{k=0}^{\infty}\frac{(\delta)_k}{k!}\left(\frac{x}{t^{\alpha}\left(1-\frac{y}{t^{\beta}}\right)}\right)^k \\
&=\left(1-\frac{y}{t^{\beta}}\right)^{-\delta}\left(1-\frac{x}{t^{\alpha}\left(1-\frac{y}{t^{\beta}}\right)}\right)^{-\delta} \\
&=\left(1-\frac{y}{t^{\beta}}-\frac{x}{t^{\alpha}}\right)^{-\delta}.
\end{align*}
Substituting this back into the integral formula for $E_{\alpha,\beta,\gamma}^{\delta}(x,y)$ obtained above, we find the result.
\end{proof}

\begin{corollary}
For $\alpha,\beta,\gamma,\delta,\omega_1,\omega_2\in\mathbb{C}$ with $\mathrm{Re}(\alpha)>0$ and $\mathrm{Re}(\beta)>0$, and $t\in\mathbb{R}$, the univariate version \eqref{5} has the following complex integral representation:
\[
t^{\gamma-1}E_{\alpha,\beta,\gamma}^{\delta}(\omega_1t^{\alpha},\omega_2t^{\beta})=\frac{1}{2\pi i}\int_H\frac{e^{tz}z^{-\gamma}}{(1-\omega_1z^{-\alpha}-\omega_2z^{-\beta})^{\delta}}\,\mathrm{d}z,
\]
where $H$ is the Hankel contour as defined in Theorem \ref{Thm:complint}.
\end{corollary}

\begin{proof}
We start from the result of Theorem \ref{Thm:complint}, replacing the variable of integration $t$ by $u$ to avoid confusion with the new variable $t$. Substitute $x=\omega_1t^{\alpha}$, $y=\omega_2t^{\beta}$ to get:
\[
E_{\alpha,\beta,\gamma}^{\delta}(\omega_1t^{\alpha},\omega_2t^{\beta})=\frac{1}{2\pi i}\int_H\frac{e^uu^{-\gamma}}{\big(1-\omega_1(\frac{t}{u})^{\alpha}-\omega_2(\frac{t}{u})^{\beta}\big)^{\delta}}\,\mathrm{d}u,
\]
and therefore
\[
t^{\gamma-1}E_{\alpha,\beta,\gamma}^{\delta}(\omega_1t^{\alpha},\omega_2t^{\beta})=\frac{1}{2\pi i}\int_H\frac{e^u(\frac{t}{u})^{\gamma}}{\big(1-\omega_1(\frac{t}{u})^{\alpha}-\omega_2(\frac{t}{u})^{\beta}\big)^{\delta}}\big(\tfrac{1}{t}\big)\,\mathrm{d}u.
\]
Making the substitution $z=\frac{u}{t}$, we obtain the stated result.
\end{proof}

The next result concerns the Laplace transform of our newly defined Mittag-Leffler type function. Note that -- in contrast with previous bivariate Mittag-Leffler functions \cite{ozarslan-kurt,kurt-ozarslan-fernandez} -- we cannot usefully calculate the double Laplace transform with respect to $x$ and $y$ of the bivariate function \eqref{3}, because the gamma function on the denominator involves both $k$ and $l$ together. Instead, we consider the univariate version \eqref{5} and calculate the Laplace transform with respect to $t$. The notation
\[
\mathcal{L}[f(t)](s)=\int_{0}^{\infty}e^{-st}f(t)\,\mathrm{d}t,\quad\quad\mathrm{Re}(s)>0,
\end{equation*}
is used for Laplace transforms.

\begin{theorem}[Laplace transform]
\label{Thm:Laplace}
For $\alpha ,\beta ,\gamma ,\delta \in \mathbb{C}$ with $\mathrm{Re}(\alpha ),\mathrm{Re}(\beta ),\mathrm{Re}(\gamma )>0$, we have
\begin{equation*}
\mathcal{L}\Big[ t^{\gamma -1}E_{\alpha ,\beta ,\gamma}^{\delta }(\omega_1t^{\alpha },\omega _{2}t^{\beta })\Big] (s)=\frac{1}{s^{\gamma }}\left( 1-\frac{\omega _{1}}{s^{\alpha }}-\frac{\omega _{2}}{s^{\beta }}\right) ^{-\delta },\quad\quad\mathrm{Re}(s)>0.
\end{equation*}
\end{theorem}

\begin{proof}
Because the double series is locally uniformly convergent, we have the right to integrate it term by term. The Laplace transform of a power function is given by
\[
\mathcal{L}\left[\frac{t^{q-1}}{\Gamma(q)}\right](s)=\frac{1}{s^q},
\]
provided that $\mathrm{Re}(q)>-1$. (This is why we introduced the extra condition $\mathrm{Re}(\gamma)>0$ as well as the standard $\mathrm{Re}(\alpha),\mathrm{Re}(\beta)>0$.) So for the Mittag-Leffler double series we have:
\begin{align*}
\mathcal{L}\left[ t^{\gamma -1}E_{\alpha ,\beta ,\gamma }^{\delta }(\omega_{1}t^{\alpha },\omega _{2}t^{\beta })\right] (s)&=\sum_{k=0}^{\infty }\sum_{l=0}^{\infty }\mathcal{L}\left[\frac{(\delta )_{k+l}\omega_{1}^{k}\omega _{2}^{l}}{\Gamma (\alpha k+\beta l+\gamma )k!l!}t^{\alpha k+\beta l+\gamma -1}\right](s) \\
&=\frac{1}{s^{\gamma }}\sum_{k=0}^{\infty }\sum_{l=0}^{\infty }\frac{(\delta )_{k+l}}{k!l!}\left( \frac{\omega _{1}}{s^{\alpha }}\right)^{k}\left( \frac{\omega _{2}}{s^{\beta }}\right) ^{l} \\
&=\frac{1}{s^{\gamma }}\sum_{k=0}^{\infty }\frac{(\delta )_{k}}{k!}\left( \frac{\omega _{1}}{s^{\alpha }}\right)^{k}\sum_{l=0}^{\infty }\frac{(\delta +k)_{l}}{l!}\left( \frac{\omega _{2}}{s^{\beta }}\right) ^{l} \\
&=\frac{1}{s^{\gamma }}\left( 1-\frac{\omega _{2}}{s^{\beta }}\right)^{-\delta }\sum_{k=0}^{\infty }\frac{(\delta )_{k}}{k!}\left( \frac{\omega_{1}}{s^{\alpha }}\right) ^{k}\left( 1-\frac{\omega _{2}}{s^{\beta }}\right)^{-k} \\
&=\frac{1}{s^{\gamma }}\left( 1-\frac{\omega _{2}}{s^{\beta }}\right)^{-\delta }\left( 1-\frac{\omega _{1}}{s^{\alpha }}\left( 1-\frac{\omega _{2}}{s^{\beta }}\right)^{-1}\right) ^{-\delta } \\
&=\frac{1}{s^{\gamma }}\left( 1-\frac{\omega _{1}}{s^{\alpha }}-\frac{\omega _{2}}{s^{\beta }}\right) ^{-\delta }.
\end{align*}
Note that the argument to simplify the double series here is the same as used in the proof of Theorem \ref{Thm:complint}. We have required extra conditions on $s$, namely $\left\vert \frac{\omega _{2}}{s^{\beta }}\right\vert <1$ and $\left\vert \frac{\omega _{1}}{s^{\alpha }}\left( 1-\frac{\omega _{2}}{s^{\beta }}\right)^{-1}\right\vert <1$, for proper convergence of the series. But these conditions can be removed at the end by analytic continuation, to give the desired result for all $s$ with $\mathrm{Re}(s)>0$.
\end{proof}

We now consider the relationship between the bivariate Mittag-Leffler function and some associated bivariate Laguerre polynomials. This work follows the approach of Prabhakar \cite{prabhakar2} for the univariate Mittag-Leffler function $E_{\alpha,\beta}^{\rho}(x)$, and similar results were found in another bivariate case in \cite{ozarslan-kurt1}.

The original Laguerre polynomials are a family of orthogonal polynomials $L_n$ defined by
\[
L_n(x)=\sum_{k=0}^n\frac{(-1)^k}{k!}\binom{n}{k}x^k
\]
or equivalently as solutions of the Laguerre differential equation $xy''+(1-x)y'+ny=0$. Also known classically are the generalised Laguerre polynomials $L_n^{(\alpha)}$ for $n\in\mathbb{Z}^+_0$ and $\alpha\in(-1,\infty)$, which are used to solve various problems of mathematical physics and quantum mechanics.

Now, we introduce a class of bivariate Laguerre polynomials and related them to the newly defined bivariate Mittag-Leffler function as follows:
\[
L_n^{\alpha,\beta,\gamma}(x,y)=\sum_{l=0}^n\sum_{k=0}^{n-l}\frac{(-n)_{k+l}x^ky^l}{\Gamma(\alpha k+\beta l+\gamma)k!l!}=E_{\alpha,\beta,\gamma}^{-n}(x,y),\quad\quad\alpha,\beta,\gamma\in\mathbb{C}.
\]
In the next theorem, we give the linear generating function for the polynomials $L_{n}^{\alpha,\beta,\gamma}(x,y)$ in terms of the bivariate Mittag-Leffler function $E_{\alpha,\beta,\gamma}^{\delta}(x,y)$.

\begin{theorem}
\label{Thm:Laguerre}
For $\alpha,\beta,\gamma,\delta\in\mathbb{C}$ with $\mathrm{Re}(\alpha),\mathrm{Re}(\beta)>0$ and $\left\vert t\right\vert <1$, we have
\begin{equation*}
\sum_{n=0}^{\infty}\frac{(\delta)_n}{n!}L_n^{\alpha,\beta,\gamma}(x,y)t^n=(1-t)^{-\delta}E_{\alpha,\beta,\gamma}^{\delta}\left( \frac{-xt}{1-t},\frac{-yt}{1-t}\right).
\end{equation*}
\end{theorem}

\begin{proof}
Using the fact that $(-n)_{k+l}=(-1)^{k+l}\frac{n!}{(n-k-l)!}$, we have
\begin{equation*}
\sum_{n=0}^{\infty}\frac{(\delta)_n}{n!}L_n^{\alpha,\beta,\gamma}(x,y)t^n=\sum_{n=0}^{\infty }\sum_{l=0}^{n}\sum_{k=0}^{n-l}\frac{(\delta)_n(-1)^{k+l}}{(n-k-l)!\Gamma (\alpha k+\beta l+\gamma )}\cdot\frac{x^ky^l}{k!l!}t^{n}
\end{equation*}
Noting that $\sum_{n=0}^{\infty}\sum_{l=0}^n=\sum_{l=0}^{\infty}\sum_{n=l}^{\infty}$ and then setting $n\rightarrow n+l$, we get
\begin{equation*}
\sum_{n=0}^{\infty}\frac{(\delta)_n}{n!}L_n^{\alpha,\beta,\gamma}(x,y)t^n=\sum_{l=0}^{\infty}\sum_{n=0}^{\infty}\sum_{k=0}^{n}\frac{(\delta)_{n+l}(-1)^{k+l}}{(n-k)!\Gamma (\alpha k+\beta l+\gamma )}\cdot\frac{x^ky^l}{k!l!}t^{n+l}.
\end{equation*}
Similarly, setting $n\rightarrow n+k$, we get
\begin{equation*}
\sum_{n=0}^{\infty}\frac{(\delta)_n}{n!}L_n^{\alpha,\beta,\gamma}(x,y)t^n=\sum_{l=0}^{\infty}\sum_{k=0}^{\infty}\sum_{n=0}^{\infty}\frac{(\delta)_{n+k+l}(-1)^{k+l}}{(n)!\Gamma (\alpha k+\beta l+\gamma )}\cdot\frac{x^ky^l}{k!l!}t^{n+k+l},
\end{equation*}
where in both cases interchanging the order of summations is guaranteed because of the uniform convergence of the series. From $(\delta)_{n+k+l}=(\delta)_{k+l}(\delta+k+l)_{n}$, we have
\begin{align*}
\sum_{n=0}^{\infty}\frac{(\delta)_n}{n!}L_n^{\alpha,\beta,\gamma}(x,y)t^n &=\sum_{l=0}^{\infty}\sum_{k=0}^{\infty}\frac{(\delta)_{k+l}}{\Gamma (\alpha k+\beta l+\gamma )}\cdot\frac{x^ky^l}{k!l!}(-t)^{k+l}\sum_{n=0}^{\infty}\frac{(\delta+k+l)_{n}}{n!}t^{n} \\
&=(1-t)^{-\delta}\sum_{l=0}^{\infty }\sum_{k=0}^{\infty }\frac{(\delta)_{k+l}}{\Gamma (\alpha k+\beta l+\gamma )k!l!}\left( \frac{-xt}{1-t}\right)^k\left( \frac{-yt}{1-t}\right)^l \\
&=(1-t)^{-\delta}E_{\alpha,\beta,\gamma}^{\delta}\left( \frac{-xt}{1-t},\frac{-yt}{1-t}\right) ,
\end{align*}
where we have used the fact $\sum_{n=0}^{\infty}\frac{(\delta+k+l)_{n}}{n!}t^{n}=(1-t)^{-\delta-k-l}$ which holds under the condition $\left\vert t\right\vert <1$.
\end{proof}

\begin{corollary}
For $\alpha,\beta,\gamma,\delta\in\mathbb{C}$ with $\mathrm{Re}(\alpha),\mathrm{Re}(\beta)>0$ and $\left\vert z\right\vert <1$, we have
\begin{equation*}
\sum_{n=0}^{\infty}\frac{(\delta)_n}{n!}L_n^{\alpha,\beta,\gamma}(t^{\alpha},t^{\beta})z^n=(1-z)^{-\delta}E_{\alpha,\beta,\gamma}^{\delta}\left(\frac{-t^{\alpha }z}{1-z},\frac{-t^{\beta }z}{1-z}\right).
\end{equation*}
\end{corollary}

\begin{proof}
This follows directly from Theorem \ref{Thm:Laguerre} when we set $x=t^{\alpha}$ and $y=t^{\beta}$.
\end{proof}

We have already discussed a connection between our Mittag-Leffler functions and fractional calculus, in Theorem \ref{Thm:FDE} and Corollary \ref{Coroll:FDE} above. In the following theorems, we calculate the general fractional differintegral (not only to the specific orders $\alpha,\beta$ matching the parameters) of our Mittag-Leffler function. Similarly to Theorem \ref{Thm:Laplace}, we must use the univariate version \eqref{5} instead of the general bivariate function \eqref{3} for best results, because the gamma function involving both $k$ and $l$ is best associated with a power function involving both $k$ and $l$.

\begin{theorem}
\label{Thm:RL:MLfn}
Let $\alpha ,\beta ,\gamma ,\delta \in\mathbb{C}$ with $\mathrm{Re}(\alpha ),\mathrm{Re}(\beta ),\mathrm{Re}(\gamma )>0$. Then the fractional differintegral of the function \eqref{5} is given by:
\begin{equation*}
\prescript{RL}{c}D_t^{\mu }\left[ (t-c)^{\gamma -1}E_{\alpha,\beta,\gamma}^{\delta}\left( \omega_1(t-c)^{\alpha },\omega_2(t-c)^{\beta }\right) \right]=(t-c)^{\gamma-\mu -1}E_{\alpha,\beta,\gamma-\mu}^{\delta}\left(\omega_1(t-c)^{\alpha},\omega_2(t-c)^{\beta }\right),
\end{equation*}
for all $mu\in\mathbb{C}$ (fractional integral if $\mathrm{Re}(\mu)<0$, or fractional derivative if $\mathrm{Re}(\mu)\geq0$).
\end{theorem}

\begin{proof}
First we recall the well-known results \cite[\S5.2]{oldham-spanier} that a uniformly convergent series of functions can be fractionally integrated term by term, and that it can be fractionally differentiated term by term provided that the series of fractional derivatives is also uniformly convergent.

Now, fractionally differintegrating the series \eqref{5} term by term gives
\begin{align*}
&\prescript{RL}{c}D_t^{\mu }\left[ (t-c)^{\gamma -1}E_{\alpha,\beta,\gamma}^{\delta}\left( \omega_1(t-c)^{\alpha },\omega_2(t-c)^{\beta }\right) \right] \\
&\hspace{2cm}=\sum_{k=0}^{\infty }\sum_{l=0}^{\infty }\frac{(\delta )_{k+l}\omega_1^k\omega_2^l}{k!l!}\prescript{RL}{c}D_t^{\mu }\left[\frac{(t-c)^{\alpha k+\beta l+\gamma -1}}{\Gamma (\alpha k+\beta l+\gamma )}\right] \\
&\hspace{2cm}=\sum_{k=0}^{\infty }\sum_{l=0}^{\infty }\frac{(\delta )_{k+l}\omega_1^k\omega_2^l}{k!l!}\cdot\frac{(t-c)^{\alpha k+\beta l+\gamma-\mu-1}}{\Gamma (\alpha k+\beta l+\gamma-\mu)} \\
&\hspace{2cm}=(t-c)^{\gamma +\mu -1}\sum_{k=0}^{\infty }\sum_{l=0}^{\infty }\frac{(\delta )_{k+l}}{\Gamma (\alpha k+\beta l+\gamma-\mu )}\frac{\left( \omega_{1}(t-c)^{\alpha }\right) ^{k}}{k!}\frac{\left( \omega _{2}(t-c)^{\beta}\right) ^{l}}{l!} \\
&\hspace{2cm}=(t-c)^{\gamma-\mu -1}E_{\alpha ,\beta ,\gamma-\mu }^{\delta }\left(\omega_1(t-c)^{\alpha },\omega_2(t-c)^{\beta }\right) ,
\end{align*}
where the resulting series is also uniformly convergent. Therefore fractionally differintegrating the whole series works as stated. Here we have used the fact \eqref{power:RL} on fractional differintegration of power functions with exponent having real part greater than $-1$; the latter restriction is why we need the extra assumption $\mathrm{Re}(\gamma)>0$ as well as the standard $\mathrm{Re}(\alpha),\mathrm{Re}(\beta)>0$.
\end{proof}

Finally, we prove a result on convolutions which will be useful in the next section.

\begin{theorem}[Convolution result]
\label{Thm:convolution}
Let $\alpha ,\beta ,\gamma_1,\gamma_2,\delta_1,\delta_2\in 
\mathbb{C}$ with $\mathrm{Re}(\alpha ),\mathrm{Re}(\beta ),\mathrm{Re}(\gamma_j)>0$. Then,
\[
\Big[t^{\gamma_1-1}E_{\alpha ,\beta ,\gamma_1}^{\delta_1}\left( \omega_1t^{\alpha },\omega_2t^{\beta }\right)\Big] \ast \Big[t^{\gamma_2-1}E_{\alpha,\beta ,\gamma_2}^{\delta_2}\left( \omega_1t^{\alpha },\omega_2t^{\beta }\right)\Big]=t^{\gamma_1+\gamma_2 -1}E_{\alpha ,\beta ,\gamma_1+\gamma_2 }^{\delta_1+\delta_2}\left( \omega_1t^{\alpha },\omega_2t^{\beta }\right).
\]
\end{theorem}

\begin{proof}
Using Theorem \ref{Thm:Laplace} and the convolution theorem for the Laplace transform, we have:
\begin{align*}
&\mathcal{L}\left\{ \Big[t^{\gamma_1-1}E_{\alpha ,\beta ,\gamma_1}^{\delta_1}\left( \omega_1t^{\alpha },\omega_2t^{\beta }\right)\Big] \ast \Big[t^{\gamma_2-1}E_{\alpha,\beta ,\gamma_2}^{\delta_2}\left( \omega_1t^{\alpha },\omega_2t^{\beta }\right)\Big]\right\} (s) \\
&=\mathcal{L}\Big[t^{\gamma_1-1}E_{\alpha ,\beta ,\gamma_1}^{\delta_1}\left( \omega_1t^{\alpha },\omega_2t^{\beta }\right)\Big](s) \mathcal{L}\Big[t^{\gamma_2-1}E_{\alpha,\beta ,\gamma_2}^{\delta_2}\left( \omega_1t^{\alpha },\omega_2t^{\beta }\right)\Big](s) \\
&=\frac{1}{s^{\gamma_1}}\left( 1-\frac{\omega_1}{s^{\alpha}}-\frac{\omega_2}{s^{\beta}}\right)^{-\delta_1}\frac{1}{s^{\gamma_2}}\left(1-\frac{\omega_1}{s^{\alpha }}-\frac{\omega_2}{s^{\beta }}\right)^{-\delta_2} \\
&=\frac{1}{s^{\gamma_1+\gamma_2}}\left( 1-\frac{\omega_1}{s^{\alpha }}-\frac{\omega_2}{s^{\beta }}\right) ^{-(\delta_1+\delta_2)} \\
&=\mathcal{L}\left\{ x^{\gamma_1+\gamma_2-1}E_{\alpha,\beta,\gamma_1+\gamma_2}^{\delta_1+\delta_2}\left(\omega_1x^{\alpha},\omega_2x^{\beta}\right)\right\} (s)
\end{align*}
Taking inverse Laplace transforms, we obtain the result.
\end{proof}

\section{The associated fractional-calculus operators} \label{Sec:op}

In this section, we define a family of fractional-calculus operators by using the bivariate Mittag-Leffler functions defined above. These operators have been discovered to emerge naturally from experimental data, and we analyse their mathematical properties here so that they can be more usefully applied in practice.

\subsection{The fractional integral operator}

\begin{definition}
\label{Def:fracint}
We define a fractional integral operator via convolution of an input function $f$ with our bivariate Mittag-Leffler function taken in the univariate form \eqref{5}, using the following notation:
\begin{equation}
\label{4}
\Big({}_c\mathfrak{I}_{\alpha,\beta,\gamma}^{\delta;\omega_1,\omega_2}f\Big)(x)=\int_c^x(x-\xi)^{\gamma-1}E_{\alpha,\beta ,\gamma}^{\delta }\left( \omega_1(x-\xi)^{\alpha},\omega_2(x-\xi)^{\beta }\right)f(\xi)\,\mathrm{d}\xi,
\end{equation}
where we assume that $x>c$ in $\mathbb{R}$ and the parameters $\alpha,\beta,\gamma,\delta,\omega_1,\omega_2\in\mathbb{C}$ satisfy $\mathrm{Re}(\alpha),\mathrm{Re}(\beta),\mathrm{Re}(\gamma)>0$. (The restrictions $\mathrm{Re}(\alpha),\mathrm{Re}(\beta)>0$ are always required for this Mittag-Leffler function, and the extra restriction $\mathrm{Re}(\gamma)>0$ is to avoid a non-integrable singularity at the endpoint $\xi=x$.) As a function space for $f$, we may use the set $L^1(c,d)$ thanks to Theorem \ref{thm2} below.
\end{definition}

\begin{remark}
In the case $\delta=0$, the integral operator \eqref{4} coincides with the original Riemann--Liouville fractional integral, that is
\begin{equation}
{}_c\mathfrak{I}_{\alpha,\beta,\gamma}^{\delta;\omega_1,\omega_2}f(x)=\prescript{RL}{c}I_x^{\gamma}f(x). \label{6}
\end{equation}
The reason for this is that when $\delta=0$, the entire double series \eqref{3} collapses to a single constant term $E_{\alpha,\beta,\gamma}^{0}(x,y)=1$.
\end{remark}

Now, we will write the integral operator (\ref{4}) in terms of Riemann--Liouville fractional integrals using a series formula. This strategy has been used before \cite{baleanu-fernandez,fernandez-baleanu-srivastava,fernandez-ozarslan-baleanu} to prove useful facts about many fractional-calculus operators by reducing them to the Riemann--Liouville case.

\begin{theorem}
\label{Thm:seriesformula}
For $\alpha ,\beta ,\gamma ,\delta ,\omega_1,\omega_2\in \mathbb{C}$ with $\mathrm{Re}(\alpha ),\mathrm{Re}(\beta ),\mathrm{Re}(\gamma )>0$, and for any function $f\in L^1(c,d)$, the fractional integral operator \eqref{4} can be written as
\begin{equation}
\label{seriesformula}
\Big({}_c\mathfrak{I}_{\alpha,\beta,\gamma}^{\delta;\omega_1,\omega_2}f\Big)(x)=\sum_{k=0}^{\infty }\sum_{l=0}^{\infty }\frac{(\delta )_{k+l}\omega_1^k\omega_2^l}{k!l!}\left(\prescript{RL}{c}I_x^{\alpha k+\beta l+\gamma }f\right) (x),
\end{equation}
where the series on the right-hand side is locally uniformly convergent.
\end{theorem}

\begin{proof}
Since the bivariate Mittag-Leffler function (\ref{3}) is convergent locally uniformly when $\mathrm{Re}(\alpha )>0$ and $\mathrm{Re}(\beta )>0$, we can interchange the summation and integration as follows:
\begin{align*}
\Big({}_c\mathfrak{I}_{\alpha,\beta,\gamma}^{\delta;\omega_1,\omega_2}f\Big)(x)&=\int_c^x(x-\xi)^{\gamma -1}E_{\alpha,\beta ,\gamma}^{\delta}\left( \omega_1(x-\xi)^{\alpha },\omega_2(x-\xi)^{\beta }\right) f(\xi)\,\mathrm{d}t \\
&=\sum_{k=0}^{\infty }\sum_{l=0}^{\infty }\frac{(\delta )_{k+l}\omega_1^k\omega _2^l}{\Gamma (\alpha k+\beta l+\gamma )k!l!}\int_c^x(x-\xi)^{\alpha k+\beta l+\gamma -1}f(\xi)\,\mathrm{d}\xi \\
&=\sum_{k=0}^{\infty }\sum_{l=0}^{\infty }\frac{(\delta )_{k+l}\omega_1^k\omega_2^l}{k!l!}\left(\prescript{RL}{c}I_x^{\alpha k+\beta l+\gamma}f\right)(x).
\end{align*}
Note that the gamma function on the denominator matches exactly with the power function in the kernel, so that we recover the Riemann--Liouville integral without an extra gamma function quotient multiplier. Note also that the condition $\mathrm{Re}(\gamma)>0$ ensures we are always dealing with Riemann--Liouville \textit{integrals} and not derivatives here.
\end{proof}


\begin{theorem} \label{thm2}
Let $\alpha ,\beta ,\gamma ,\delta ,\omega_1,\omega_2\in \mathbb{C}$ with $\mathrm{Re}(\alpha ),\mathrm{Re}(\beta ),\mathrm{Re}(\gamma )>0$. The fractional integral operator ${}_c\mathfrak{I}_{\alpha,\beta,\gamma}^{\delta;\omega_1,\omega_2}$ is bounded on the space $L^{1}(c,d)$, such that
\begin{equation*}
\left\Vert{}_c\mathfrak{I}_{\alpha,\beta,\gamma}^{\delta;\omega_1,\omega_2}f\right\Vert_1\leq A\left\Vert f\right\Vert_1,\quad\quad f\in L^1(c,d),
\end{equation*}
where $A$ is a constant (independent of $f$) given by
\begin{equation*}
A=\sum_{k=0}^{\infty }\sum_{l=0}^{\infty }\frac{\left\vert (\delta )_{k+l}\right\vert|\omega_1|^k|\omega_2|^l}{\mathrm{Re}(\alpha k+\beta l+\gamma)\left\vert \Gamma (\alpha k+\beta l+\gamma )\right\vert }\cdot\frac{(d-c)^{\mathrm{Re}(\alpha k+\beta l+\gamma)}}{k!l!}.
\end{equation*}
\end{theorem}

\begin{proof}
Let $f\in L^1(c,d)$. By using the definitions of the operator \eqref{4} and the $L^1$ space, we have
\begin{align*}
\left\Vert{}_c\mathfrak{I}_{\alpha,\beta,\gamma}^{\delta;\omega_1,\omega_2}f\right\Vert_1 &=\int_c^d\left\vert\int_c^x(x-\xi)^{\gamma -1}E_{\alpha ,\beta,\gamma}^{\delta }\left(\omega_1(x-\xi)^{\alpha },\omega_2(x-\xi)^{\beta }\right)f(t)\,\mathrm{d}\xi\right\vert\,\mathrm{d}x \\
&\leq \int_c^d\int_\xi^d(x-\xi)^{\mathrm{Re}(\gamma )-1}\left\vert E_{\alpha ,\beta,\gamma}^{\delta }\left( \omega_1(x-\xi)^{\alpha },\omega_2(x-\xi)^{\beta }\right) \right\vert \left\vert f(\xi)\right\vert\,\mathrm{d}x\,\mathrm{d}\xi \\
&=\int_c^d\int_{0}^{d-\xi}t^{\mathrm{Re}(\gamma )-1}\left\vert E_{\alpha,\beta,\gamma}^{\delta }\left( \omega_1t^{\alpha },\omega_2t^{\beta}\right) \right\vert \left\vert f(\xi)\right\vert \,\mathrm{d}t\,\mathrm{d}\xi \\
&\leq \int_c^d\int_{0}^{d-c}t^{\mathrm{Re}(\gamma )-1}\left\vert \sum_{k=0}^{\infty }\sum_{l=0}^{\infty }\frac{(\delta )_{k+l}}{\Gamma(\alpha k+\beta l+\gamma )}\cdot\frac{\omega_1^k\omega_2^l}{k!l!}\;t^{\alpha k+\beta l}\right\vert \left\vert f(\xi)\right\vert \,\mathrm{d}t\,\mathrm{d}\xi \\
&\leq \sum_{k=0}^{\infty }\sum_{l=0}^{\infty }\frac{\left\vert (\delta)_{k+l}\right\vert\left\vert \omega_1\right\vert^k\left\vert \omega_2\right\vert^l}{\left\vert \Gamma (\alpha k+\beta l+\gamma )\right\vert k!l!}\int_{0}^{d-c}t^{\mathrm{Re}(\alpha k+\beta l+\gamma)-1}\,\mathrm{d}t\left\Vert f\right\Vert_1 \\
&=\sum_{k=0}^{\infty }\sum_{l=0}^{\infty }\frac{\left\vert (\delta)_{k+l}\right\vert\left\vert \omega_1\right\vert^k\left\vert \omega_2\right\vert^l}{\left\vert \Gamma (\alpha k+\beta l+\gamma )\right\vert k!l!}\cdot\frac{(d-c)^{\mathrm{Re}(\alpha k+\beta l+\gamma)}}{\mathrm{Re}(\alpha k+\beta l+\gamma)}\left\Vert f\right\Vert _1,
\end{align*}
which is the end of the proof.
\end{proof}

\begin{corollary}
For $\alpha,\beta,\gamma,\delta,\omega_1,\omega_2\in\mathbb{C}$ with $\mathrm{Re}(\alpha),\mathrm{Re}(\beta),\mathrm{Re}(\gamma)>0$, the integral operator \eqref{4} with bivariate Mittag-Leffler kernel interacts with the standard Riemann--Liouville fractional integral or derivative in the following way.
\begin{align*}
\left(\prescript{RL}{c}I_x^{\mu}\left({}_c\mathfrak{I}_{\alpha,\beta,\gamma}^{\delta;\omega_1,\omega_2}f\right)\right)(x)&=\Big({}_c\mathfrak{I}_{\alpha,\beta,\gamma+\mu}^{\delta;\omega_1,\omega_2}f\Big)(x),\quad\quad\mu\in\mathbb{C}; \\
\left( {}_c\mathfrak{I}_{\alpha,\beta,\gamma}^{\delta;\omega_1,\omega_2}\left(\prescript{RL}{c}I_x^{\mu}f\right)\right)(x)&=\Big({}_c\mathfrak{I}_{\alpha,\beta,\gamma+\mu}^{\delta;\omega_1,\omega_2}f\Big)(x),\quad\quad\mathrm{Re}(\mu)>0.
\end{align*}
(The conditions on the right mean that the first identity is valid for all fractional differintegrals while the second is valid only for fractional integrals and not derivatives.)
\end{corollary}

\begin{proof}
It is well known that, in the Riemann--Liouville model, a semigroup property is valid when we take a differintegral of an integral:
\begin{equation}
\label{RLsemigroup}
\prescript{RL}{c}I^{\mu}_x\prescript{RL}{c}I^{\nu}_xf(x)=\prescript{RL}{c}I^{\mu+\nu}_xf(x),\quad\quad\mu,\nu\in\mathbb{C},\mathrm{Re}(\nu)>0.
\end{equation}
Using also the facts cited in the proof of Theorem \ref{Thm:RL:MLfn} about fractional differintegration of series, we obtain the result directly from the series formula \eqref{seriesformula}.
\end{proof}

In the next theorem, we state the semigroup property of the operator \eqref{4}. This is a very important issue to consider for any fractional-calculus operator, since we usually expect an integral of an integral to be an integral, but sometimes it may also be useful to break this condition. Also, the semigroup property will enable us to define an inverse to the fractional integral operator, and hence to find the corresponding fractional derivative operators. We use three different proof techniques to prove this theorem.

\begin{theorem}[Semigroup property in $\gamma$ and $\delta$ parameters]
\label{Thm:semigroup}
Let $\alpha,\beta,\gamma_1,\gamma_2,\delta_1,\delta_2,\omega_1,\omega_2\in\mathbb{C}$ with $\mathrm{Re}(\alpha)>0$, $\mathrm{Re}(\beta)>0$, $\mathrm{Re}(\gamma_j)>0$. For any $f\in L^1(c,d)$, we have the following semigroup property for the  fractional integral operator \eqref{4}:
\begin{equation}
\left({}_c\mathfrak{I}_{\alpha,\beta,\gamma_1}^{\delta_1;\omega_1,\omega_2}{}_c\mathfrak{I}_{\alpha,\beta,\gamma_2}^{\delta_2;\omega_1,\omega_2}f\right)(x)=\left({}_c\mathfrak{I}_{\alpha,\beta,\gamma_1+\gamma_2}^{\delta_1+\delta_2;\omega_1,\omega_2}f\right)(x). \label{7}
\end{equation}
In particular, setting $\delta_2=-\delta_1$ and using relation \eqref{6}, we find that
\begin{equation}
\left({}_c\mathfrak{I}_{\alpha,\beta,\gamma_1}^{\delta;\omega_1,\omega_2}{}_c\mathfrak{I}_{\alpha,\beta,\gamma_2}^{-\delta;\omega_1,\omega_2}f\right)(x)=\prescript{RL}{c}I_x^{\gamma_1+\gamma_2}f(x).  \label{8}
\end{equation}
\end{theorem}

\begin{proof}[Proof using convolution relation]
Here we use direct integration and the result of Theorem \ref{Thm:convolution} (which was proved using Laplace transforms):
\begin{align*}
\left({}_c\mathfrak{I}_{\alpha,\beta,\gamma_1}^{\delta_1;\omega_1,\omega_2}{}_c\mathfrak{I}_{\alpha,\beta,\gamma_2}^{\delta_2;\omega_1,\omega_2}f\right)(x)&=\int_c^x(x-t)^{\gamma_1-1}E_{\alpha ,\beta ,\gamma_1}^{\delta_1}\left(\omega_1(x-t)^{\alpha },\omega_2(x-t)^{\beta }\right){}_c\mathfrak{I}_{\alpha,\beta,\gamma_2}^{\delta_2;\omega_1,\omega_2}f(t)\mathrm{d}t \\
&\hspace{-4cm}=\int_c^x\int_c^t(x-t)^{\gamma_1-1}(t-\tau )^{\gamma_2-1}E_{\alpha,\beta ,\gamma_1}^{\delta_1}\left( \omega_1(x-t)^{\alpha },\omega_2(x-t)^{\beta }\right) E_{\alpha ,\beta ,\gamma_2}^{\delta_2 }\left( \omega_1(t-\tau )^{\alpha },\omega_2(t-\tau )^{\beta }\right)f(\tau)\mathrm{d}\tau\mathrm{d}t \\
&\hspace{-4cm}=\int_c^x\int_{\tau}^x(x-t)^{\gamma_1-1}(t-\tau )^{\gamma_2-1}E_{\alpha ,\beta ,\gamma_1}^{\delta_1}\left( \omega_1(x-t)^{\alpha},\omega_2(x-t)^{\beta }\right) E_{\alpha ,\beta ,\gamma_2}^{\delta_2}\left(\omega_1(t-\tau )^{\alpha },\omega_2(t-\tau )^{\beta }\right) f(\tau )\mathrm{d}t\mathrm{d}\tau  \\
&\hspace{-4cm}=\int_c^xf(\tau)\int_{0}^{x-\tau }(x-\tau-u)^{\gamma_1-1}u^{\gamma_2-1}E_{\alpha ,\beta ,\gamma_1}^{\delta_1}\left( \omega_1(x-\tau-u)^{\alpha},\omega_2(x-\tau-u)^{\beta }\right) E_{\alpha ,\beta ,\gamma_2}^{\delta_2}\left( \omega_1u^{\alpha },\omega_2u^{\beta }\right)\mathrm{d}u\mathrm{d}\tau  \\
&\hspace{-4cm}=\int_c^x\Big[t^{\gamma_1-1}E_{\alpha ,\beta ,\gamma_1}^{\delta_1}\left( \omega_1t^{\alpha },\omega_2t^{\beta }\right)\Big] \ast \Big[t^{\gamma_2-1}E_{\alpha,\beta ,\gamma_2}^{\delta_2}\left( \omega_1t^{\alpha },\omega_2t^{\beta }\right)\Big]_{t=x-\tau}f(\tau)\mathrm{d}\tau  \\
&\hspace{-4cm}=\int_c^x(x-\tau)^{\gamma_1+\gamma_2-1}E_{\alpha ,\beta ,\gamma_1+\gamma_2}^{\delta_1+\delta_2}\left(\omega_1(x-\tau)^{\alpha },\omega_2(x-\tau)^{\beta }\right)f(\tau )\mathrm{d}\tau  \\
&\hspace{-4cm}=\left({}_c\mathfrak{I}_{\alpha,\beta,\gamma_1+\gamma_2}^{\delta_1+\delta_2;\omega_1,\omega_2}f\right)(x),
\end{align*}
which is the desired result.
\end{proof}

\begin{proof}[Proof using series formula]
Here we use the result of Theorem \ref{Thm:seriesformula}, and also the fact \eqref{RLsemigroup} about the semigroup property of Riemann--Liouville integrals.
\begin{align*}
\left({}_c\mathfrak{I}_{\alpha,\beta,\gamma_1}^{\delta_1;\omega_1,\omega_2}{}_c\mathfrak{I}_{\alpha,\beta,\gamma_2}^{\delta_2;\omega_1,\omega_2}f\right)(x)&=\sum_{k=0}^{\infty }\sum_{l=0}^{\infty }\frac{(\delta_1)_{k+l}\omega_1^k\omega_2^l}{k!l!}\left(\prescript{RL}{c}I_x^{\alpha k+\beta l+\gamma_1}{}_c\mathfrak{I}_{\alpha,\beta,\gamma_2}^{\delta_2;\omega_1,\omega_2}f\right)(x) \\
&\hspace{-4cm}=\sum_{k=0}^{\infty }\sum_{l=0}^{\infty }\frac{(\delta_1)_{k+l}\omega_1^k\omega_2^l}{k!l!}\prescript{RL}{c}I_x^{\alpha k+\beta l+\gamma}\sum_{i=0}^{\infty }\sum_{j=0}^{\infty }\frac{(\delta_2)_{i+j}\omega_1^i\omega_2^j}{i!j!}\prescript{RL}{c}I_x^{\alpha i+\beta j+\gamma_2}f(x) \\
&\hspace{-4cm}=\sum_{k=0}^{\infty }\sum_{l=0}^{\infty }\sum_{i=0}^{\infty}\sum_{j=0}^{\infty }\frac{(\delta_1)_{k+l}(\delta_2)_{i+j}\omega_1^{k+i}\omega_2^{l+j}}{k!l!i!j!}\prescript{RL}{c}I_x^{\alpha k+\beta l+\gamma_1}\prescript{RL}{c}I_x^{\alpha i+\beta j+\gamma_2}f(x) \\
&\hspace{-4cm}=\sum_{k=0}^{\infty }\sum_{l=0}^{\infty }\sum_{i=0}^{\infty}\sum_{j=0}^{\infty }\frac{\Gamma(\delta_1+k+l)\Gamma(\delta_2+i+j)}{\Gamma(\delta_1)\Gamma(\delta_2)k!l!i!j!}\omega_1^{k+i}\omega_2^{l+j}\prescript{RL}{c}I_x^{\alpha(k+i)+\beta(l+j)+\gamma_1+\gamma_2}f(x) \\
&\hspace{-4cm}=\sum_{k=0}^{\infty }\sum_{l=0}^{\infty }\sum_{i=0}^{\infty}\sum_{j=0}^{\infty }\frac{B(\delta_1+k+l,\delta_2+i+j)\Gamma(\delta_1+\delta_2+k+l+i+j)}{B(\delta_1,\delta_2)\Gamma(\delta_1+\delta_2)k!l!i!j!}\omega_1^{k+i}\omega_2^{l+j}\prescript{RL}{c}I_x^{\alpha(k+i)+\beta(l+j)+\gamma_1+\gamma_2}f(x) \\
&\hspace{-4cm}=\sum_{p=0}^{\infty }\sum_{q=0}^{\infty }\left[ \sum_{i+k=p}\sum_{j+l=q}\frac{B(\delta_1+k+l,\delta_2+i+j)p!q!}{B(\delta_1,\delta_2)k!l!i!j!}\right]\frac{\Gamma(\delta_1+\delta_2+p+q)}{\Gamma (\delta_1+\delta_2)p!q!}\omega_1^p\omega_2^q\prescript{RL}{c}I_x^{p\alpha +q\beta +\gamma_1+\gamma_2}f(x).
\end{align*}
It remains to simplify the expression in square brackets. We use
the fact that $\sum_{m+n=k}\frac{B(\rho_1+n,\rho_2+m)k!}{B(\rho_1,\rho_2)n!m!}=1$, proved in \cite[Theorem 2.9]{fernandez-baleanu-srivastava}, and apply this identity twice in the following analysis:
\begin{align*}
&{\color{white}=}\sum_{i+k=p}\sum_{j+l=q}\frac{B(\delta_1+k+l,\delta_2+i+j)p!q!}{B(\delta_1,\delta_2)k!l!i!j!} \\
&=\sum_{i+k=p}\frac{p!}{k!i!}\left[ \frac{B(\delta_1+k,\delta_2+i)}{B(\delta_1,\delta_2)}\right]\sum_{j+l=q}\frac{B((\delta_1+k)+l,(\delta_2+i)+j)q!}{B(\delta_1+k,\delta_2+i)l!j!}  \\
&=\sum_{i+k=p}\frac{B(\delta_1+k,\delta_2+i)p!}{k!i!B(\delta_1,\delta_2)}(1)  \\
&=1.
\end{align*}
Therefore, we have 
\begin{align*}
\left({}_c\mathfrak{I}_{\alpha,\beta,\gamma_1}^{\delta_1;\omega_1,\omega_2}{}_c\mathfrak{I}_{\alpha,\beta,\gamma_2}^{\delta_2;\omega_1,\omega_2}f\right)(x)&=\sum_{p=0}^{\infty}\sum_{q=0}^{\infty}\frac{\Gamma(\delta_1+\delta_2+p+q)}{\Gamma (\delta_1+\delta_2)p!q!}\omega_1^p\omega_2^q\prescript{RL}{c}I_x^{p\alpha +q\beta +\gamma_1+\gamma_2}f(x) \\
&=\sum_{p=0}^{\infty }\sum_{q=0}^{\infty }\frac{(\delta_1+\delta_2)_{p+q}}{p!q!}\prescript{RL}{c}I_x^{p\alpha +q\beta +\gamma_1+\gamma_2}f(x) \\
&=\left({}_c\mathfrak{I}_{\alpha,\beta,\gamma_1+\gamma_2}^{\delta_1+\delta_2;\omega_1,\omega_2}f\right) (x),
\end{align*}
which is the end of the proof.
\end{proof}

\begin{proof}[Proof using Laplace transforms, in the case $c=0$]
Since the integral operator \eqref{4} is defined by a convolution of the input function $f(x)$ with the function $x^{\gamma-1}E_{\alpha,\beta,\gamma}^{\delta}(\omega_1x^{\alpha},\omega_2x^{\beta})$ whose Laplace transform is calculated in Theorem \ref{Thm:Laplace}, we can use the convolution theorem for Laplace transforms to deduce that
\begin{equation}
\label{Laplace:intop}
\mathcal{L}\left[{}_0\mathfrak{I}_{\alpha,\beta,\gamma}^{\delta;\omega_1,\omega_2}f\right](s)=s^{-\gamma}\Big(1-\omega_1s^{-\alpha}-\omega_2s^{-\beta}\Big) ^{-\delta }\mathcal{L}\big[f\big],\quad\quad\mathrm{Re}(s)>0,
\end{equation}
for any parameters $\alpha,\beta,\gamma,\delta,\omega_1,\omega_2\in\mathbb{C}$ satisfying the usual conditions $\mathrm{Re}(\alpha),\mathrm{Re}(\beta),\mathrm{Re}(\gamma)>0$.

We use this fact three times in the following analysis:
\begin{align*}
\mathcal{L}\left\{ {}_0\mathfrak{I}_{\alpha,\beta,\gamma_1}^{\delta_1;\omega_1,\omega_2}{}_0\mathfrak{I}_{\alpha,\beta,\gamma_2}^{\delta_2;\omega_1,\omega_2}f \right\} (s)&=s^{-\gamma_1}\Big(1-\omega_1s^{-\alpha}-\omega_2s^{-\beta}\Big) ^{-\delta_1}\mathcal{L}\left\{ {}_0\mathfrak{I}_{\alpha,\beta,\gamma_2}^{\delta_2;\omega_1,\omega_2}f \right\} (s) \\
&=s^{-\gamma_1}\Big(1-\omega_1s^{-\alpha}-\omega_2s^{-\beta}\Big) ^{-\delta_1}s^{-\gamma_2}\Big(1-\omega_1s^{-\alpha}-\omega_2s^{-\beta}\Big) ^{-\delta_2}\mathcal{L}\left\{f\right\} (s) \\
&=s^{-\gamma_1-\gamma_2}\Big(1-\omega_1s^{-\alpha}-\omega_2s^{-\beta}\Big) ^{-\delta_1-\delta_2}\mathcal{L}\left\{f\right\} (s) \\
&=\mathcal{L}\left\{ {}_c\mathfrak{I}_{\alpha,\beta,\gamma_1+\gamma_2}^{\delta_1+\delta_2;\omega_1,\omega_2}f \right\} (s).
\end{align*}
Taking the inverse Laplace transform on both sides, we get
\begin{equation*}
\left({}_0\mathfrak{I}_{\alpha,\beta,\gamma_1}^{\delta_1;\omega_1,\omega_2}{}_0\mathfrak{I}_{\alpha,\beta,\gamma_2}^{\delta_2;\omega_1,\omega_2}f \right) (x)=\left( {}_0\mathfrak{I}_{\alpha,\beta,\gamma_1+\gamma_2}^{\delta_1+\delta_2;\omega_1,\omega_2}f \right) (x).
\end{equation*}
\end{proof}

\subsection{The fractional derivative operators}

In the following Theorem, we find the left inverse operator for the integral operator \eqref{4}. This will enable us to define operators of fractional differentiation to accompany that of fractional integration defined by \eqref{4}.

\begin{theorem}[Left inverse for the fractional integral operator] \label{thm1}
Let $\alpha ,\beta ,\gamma ,\delta ,\omega_1,\omega_2\in \mathbb{C}$ and $\mathrm{Re}(\alpha)>0$, $\mathrm{Re}(\beta)>0$, $\mathrm{Re}(\gamma)>0$. Then, for any $\zeta\in\mathbb{C}$ with $\mathrm{Re}(\zeta)>0$, the following defines a left inverse operator to the operator ${}_c\mathfrak{I}_{\alpha,\beta,\gamma}^{\delta;\omega_1,\omega_2}$ on the space $L^{1}(c,d)$:
\begin{align}
\label{leftinverse:defn} \left( {}_c\mathfrak{D}_{\alpha,\beta,\gamma}^{\delta;\omega_1,\omega_2}\varphi\right) (x)&=\left(\prescript{RL}{c}D_x^{\gamma +\zeta }{}_c\mathfrak{I}_{\alpha,\beta,\zeta}^{-\delta;\omega_1,\omega_2}\varphi\right) (x) \\
\label{leftinverse:series} &=\sum_{k=0}^{\infty }\sum_{l=0}^{\infty }\frac{(-\delta )_{k+l}\omega_1^k\omega_2^l}{k!l!}\left(\prescript{RL}{c}I_x^{\alpha k+\beta l-\gamma }\varphi\right) (x),
\end{align}
defined for any $\varphi\in L^1(c,d)$ such that $\prescript{RL}{c}D^{\gamma}_x\varphi(x)$ exists.
\end{theorem}

\begin{proof}
Let us take a function $f\in L^{1}(c,d)$ and denote
\begin{equation*}
\varphi (x):=\left({}_c\mathfrak{I}_{\alpha,\beta,\gamma}^{\delta;\omega_1,\omega_2}f\right)(x).
\end{equation*}
Since the operator ${}_c\mathfrak{I}_{\alpha,\beta,\gamma}^{\delta;\omega_1,\omega_2}$ is bounded on the space $L^{1}(c,d)$ from Theorem \ref{thm2}, we conclude that $\varphi \in L^{1}(c,d)$. Our aim here is to recover $f$ in terms of $\varphi$.

First of all, we apply the operator ${}_c\mathfrak{I}_{\alpha,\beta,\zeta}^{-\delta;\omega_1,\omega_2}$ to $\varphi$, for an arbitrary $\zeta\in\mathbb{C}$ with $\mathrm{Re}(\zeta)>0$, and use the Theorem \ref{Thm:semigroup}, specifically the special case \eqref{8}:
\begin{equation*}
{}_c\mathfrak{I}_{\alpha,\beta,\zeta}^{-\delta;\omega_1,\omega_2}\varphi (x)=\left( {}_c\mathfrak{I}_{\alpha,\beta,\zeta}^{-\delta;\omega_1,\omega_2}{}_c\mathfrak{I}_{\alpha,\beta,\gamma}^{\delta;\omega_1,\omega_2}f\right) (x)=\left(\prescript{RL}{c}I_x^{\gamma+\zeta }f\right) (x).
\end{equation*}
And the left inverse of a Riemann--Liouville fractional integral is always the corresponding Riemann--Liouville fractional derivative \cite[Theorem 2.4]{samko-kilbas-marichev}. So
\begin{equation*}
\prescript{RL}{c}D_x^{\gamma +\zeta }{}_c\mathfrak{I}_{\alpha,\beta,\zeta}^{-\delta;\omega_1,\omega_2}\varphi (x)=\prescript{RL}{c}D_x^{\gamma+\zeta }\prescript{RL}{c}I_x^{\gamma +\zeta }f(x)=f(x),
\end{equation*}
which means ${}_c\mathfrak{D}_{\alpha,\beta,\gamma}^{\delta;\omega_1,\omega_2}\varphi=f$, where the $\mathfrak{D}$ operator is defined by \eqref{leftinverse:defn}. So we have shown that this operator is the left inverse of ${}_c\mathfrak{I}_{\alpha,\beta,\gamma}^{\delta;\omega_1,\omega_2}$ on the space $L^{1}(c,d)$. It remains to prove the series expression \eqref{leftinverse:series} for the same operator, which follows from Theorem \ref{Thm:seriesformula}:
\begin{align*}
\left( {}_c\mathfrak{D}_{\alpha,\beta,\gamma}^{\delta;\omega_1,\omega_2}\varphi\right) (x)&=\left(\prescript{RL}{c}D_x^{\gamma +\zeta }{}_c\mathfrak{I}_{\alpha,\beta,\zeta}^{-\delta;\omega_1,\omega_2}\varphi\right) (x) \\
&=\prescript{RL}{c}D_x^{\gamma +\zeta }\left(\sum_{k=0}^{\infty }\sum_{l=0}^{\infty }\frac{(-\delta )_{k+l}\omega_1^k\omega_2^l}{k!l!}\left(\prescript{RL}{c}I_x^{\alpha k+\beta l+\zeta}\varphi\right)(x)\right) \\
&=\sum_{k=0}^{\infty }\sum_{l=0}^{\infty }\frac{(-\delta )_{k+l}\omega_1^k\omega_2^l}{k!l!}\Big(\prescript{RL}{c}D_x^{\gamma +\zeta }\prescript{RL}{c}I_x^{\alpha k+\beta l+\zeta}\varphi\Big)(x) \\
&=\sum_{k=0}^{\infty }\sum_{l=0}^{\infty }\frac{(-\delta )_{k+l}\omega_1^k\omega_2^l}{k!l!}\left(\prescript{RL}{c}I_x^{\alpha k+\beta l-\gamma}\varphi\right)(x),
\end{align*}
provided that this series converges, where in the last step we used the semigroup property \eqref{RLsemigroup} for differintegrals of integrals in the Riemann--Liouville model.

For existence of all terms in the series, we already have $\varphi\in L^1(c,d)$ so that all its fractional integrals are defined, so we just need the fractional derivatives up to order $\gamma$ to exist also. We also need to check convergence of the series, since (as discussed in the proof of Theorem \ref{Thm:RL:MLfn}) this is a required condition to be able to apply fractional derivatives term by term. All but finitely many terms in the series are fractional integrals not derivatives ($\mathrm{Re}(\alpha k+\beta l-\gamma>0$ for all except small $k,l$), so by shifting the value of $\gamma$, convergence of the series is equivalent to convergence of the series \eqref{seriesformula} which we already know is locally uniformly convergent.
\end{proof}

It is important to notice that the series formula \eqref{leftinverse:series} does not contain the parameter $\zeta$; thus, the definition of ${}_c\mathfrak{D}_{\alpha,\beta,\gamma}^{\delta;\omega_1,\omega_2}$ is independent of the choice of $\zeta$. This is important because it validates the notation: we chose an arbitrary number $\zeta$ to define the operator, but in the end the choice of this number does not matter.

The most natural choice of $\zeta$ is to make $\gamma+\zeta$ (the order of differentiation) a natural number, so that \eqref{leftinverse:defn} can be defined easily without using fractional derivatives. Therefore, we propose the following definition.

\begin{definition}
\label{Def:fracderivRL}
The fractional derivative operator of Riemann--Liouville type corresponding to the fractional integral operator defined in Definition \ref{Def:fracint} is given by the left inverse operator defined in the previous theorem, namely:
\begin{equation}
\label{fracderivRL}
\left( {}_c\mathfrak{D}_{\alpha,\beta,\gamma}^{\delta;\omega_1,\omega_2}f\right) (x)=\frac{\mathrm{d}^n}{\mathrm{d}x^n}\left({}_c\mathfrak{I}_{\alpha,\beta,n-\gamma}^{-\delta;\omega_1,\omega_2}f(x)\right),\quad\quad n:=\lfloor\mathrm{Re}(\gamma)\rfloor+1,
\end{equation}
where the parameters $\alpha,\beta,\gamma,\delta,\omega_1,\omega_2\in\mathbb{C}$ satisfy $\mathrm{Re}(\alpha),\mathrm{Re}(\beta)>0$ and $\mathrm{Re}(\gamma)\geq0$, and where $f$ is fractionally differentiable to order at least $\gamma$. (The natural number $n$ is defined to make sure that $\mathrm{Re}(n-\gamma)>0$ so that the integral operator in \eqref{fracderivRL} is well-defined.)
\end{definition}

The relationship between Definition \ref{Def:fracint} and Definition \ref{Def:fracderivRL} is analogous to that between the original Riemann--Liouville integral \eqref{RLorig:defint} and its extension the Riemann--Liouville derivative \eqref{RLorig:defderiv}. This extension can be justified in two ways: firstly, the fractional derivative is the left inverse of the fractional integral (the derivative of the corresponding integral is always the original function, in the Riemann--Liouville model); secondly, the unique analytic continuation in $\nu$ of the fractional integral $\prescript{RL}{c}I^{\alpha}_xf(x)$ from the domain $\mathrm{Re}(\alpha)>0$ is given by $\prescript{RL}{c}I^{\alpha}_xf(x)=\prescript{RL}{c}D^{-\alpha}_xf(x)$ for $\mathrm{Re}(\alpha)\leq0$.

In our case, Theorem \ref{thm1} shows that ${}_c\mathfrak{D}_{\alpha,\beta,\gamma}^{\delta;\omega_1,\omega_2}$ is the left inverse of ${}_c\mathfrak{I}_{\alpha,\beta,\gamma}^{\delta;\omega_1,\omega_2}$. As for analytic continuation, the series formula \eqref{leftinverse:series} shows that, if we define ${}_c\mathfrak{I}_{\alpha,\beta,\gamma}^{\delta;\omega_1,\omega_2}$ for all $\gamma\in\mathbb{C}$ using analytic continuation from the domain $\mathrm{Re}(\gamma)>0$, then we have ${}_c\mathfrak{D}_{\alpha,\beta,\gamma}^{\delta;\omega_1,\omega_2}={}_c\mathfrak{I}_{\alpha,\beta,-\gamma}^{-\delta;\omega_1,\omega_2}$ for $\mathrm{Re}(\gamma)\leq0$. This reflects the fact that the semigroup property of Theorem \ref{Thm:semigroup} is in both $\gamma$ and $\delta$, so both of them must be negated for the inverse.

Similarly to Definition \ref{Def:fracderivRL} for the fractional derivative operator of Riemann--Liouville type, we may also define one of Caputo type.

\begin{definition}
\label{Def:fracderivC}
The fractional derivative operator of Caputo type corresponding to the fractional integral operator defined in Definition \ref{Def:fracint} is given by:
\begin{equation}
\label{fracderivC}
\left( {}_c\mathfrak{C}_{\alpha,\beta,\gamma}^{\delta;\omega_1,\omega_2}f\right) (x)={}_c\mathfrak{I}_{\alpha,\beta,n-\gamma}^{-\delta;\omega_1,\omega_2}\left(\frac{\mathrm{d}^n}{\mathrm{d}x^n}f(x)\right),\quad\quad n:=\lfloor\mathrm{Re}(\gamma)\rfloor+1,
\end{equation}
where the parameters $\alpha,\beta,\gamma,\delta,\omega_1,\omega_2\in\mathbb{C}$ satisfy $\mathrm{Re}(\alpha),\mathrm{Re}(\beta)>0$ and $\mathrm{Re}(\gamma)\geq0$, and where $f\in C^n(c,d)$ such that $f^{(n)}\in L^1(c,d)$.
\end{definition}

\begin{theorem}
\label{Thm:inversion}
The fractional-calculus operators defined in this paper have the following composition relations:
\begin{align}
\label{inversion:RLI} {}_c\mathfrak{D}_{\alpha,\beta,\gamma}^{\delta;\omega_1,\omega_2}{}_c\mathfrak{I}_{\alpha,\beta,\gamma}^{\delta;\omega_1,\omega_2}f(x)&=f(x); \\
\label{inversion:IC} {}_c\mathfrak{I}_{\alpha,\beta,\gamma}^{\delta;\omega_1,\omega_2}{}_c\mathfrak{C}_{\alpha,\beta,\gamma}^{\delta;\omega_1,\omega_2}f(x)&=\left(\int_c^x\right)^n\left(\frac{\mathrm{d}}{\mathrm{d}x}\right)^nf(x)=f(x)-\sum_{k=0}^{n-1}\frac{(x-c)^kf^{(k)}(c)}{k!},
\end{align}
where $\alpha,\beta,\gamma,\delta,\omega_1,\omega_2\in\mathbb{C}$ satisfy $\mathrm{Re}(\alpha),\mathrm{Re}(\beta),\mathrm{Re}(\gamma)>0$, and in the second identity $n:=\lfloor\mathrm{Re}(\gamma)\rfloor+1$ as in Definitions \ref{Def:fracderivRL} and \ref{Def:fracderivC}.
\end{theorem}

\begin{proof}
The first identity is already proved by Theorem \ref{thm1}: the $\mathfrak{D}$ operator is the left inverse of the $\mathfrak{I}$ operator.

For the second identity, we use the definition \eqref{fracderivC} of the $\mathfrak{C}$ operator together with the composition property \eqref{8} for the $\mathfrak{I}$ operator:
\begin{align*}
{}_c\mathfrak{I}_{\alpha,\beta,\gamma}^{\delta;\omega_1,\omega_2}{}_c\mathfrak{C}_{\alpha,\beta,\gamma}^{\delta;\omega_1,\omega_2}f(x)&={}_c\mathfrak{I}_{\alpha,\beta,\gamma}^{\delta;\omega_1,\omega_2}{}_c\mathfrak{I}_{\alpha,\beta,n-\gamma}^{-\delta;\omega_1,\omega_2}\left(\frac{\mathrm{d}^n}{\mathrm{d}x^n}f(x)\right) \\
&={}_c\mathfrak{I}_{\alpha,\beta,\gamma+(n-\gamma)}^{\delta+(-\delta);\omega_1,\omega_2}\left(\frac{\mathrm{d}^n}{\mathrm{d}x^n}f(x)\right)={}_c\mathfrak{I}_{\alpha,\beta,n}^{0;\omega_1,\omega_2}\left(\frac{\mathrm{d}^n}{\mathrm{d}x^n}f(x)\right) \\
&=\prescript{RL}{c}I^n_x\left(\frac{\mathrm{d}^n}{\mathrm{d}x^n}f(x)\right),
\end{align*}
which is the desired result written in different notation.
\end{proof}

\begin{theorem}
\label{Thm:RLCrelation}
The fractional derivative operators of Riemann--Liouville and Caputo type have the following relationship:
\begin{equation}
\label{relation:RLC}
\left( {}_c\mathfrak{C}_{\alpha,\beta,\gamma}^{\delta;\omega_1,\omega_2}f\right) (x)=\left( {}_c\mathfrak{D}_{\alpha,\beta,\gamma}^{\delta;\omega_1,\omega_2}f\right)(x)-\sum_{j=0}^{n-1}(x-c)^{-\gamma+j}E_{\alpha,\beta,-\gamma+j+1}^{-\delta}\Big(\omega_1(x-c)^{\alpha},\omega_2(x-c)^{\beta}\Big)f^{(j)}(c),
\end{equation}
where $\alpha,\beta,\gamma,\delta,\omega_1,\omega_2\in\mathbb{C}$ satisfy $\mathrm{Re}(\alpha),\mathrm{Re}(\beta),\mathrm{Re}(\gamma)>0$, and $n:=\lfloor\mathrm{Re}(\gamma)\rfloor+1$ as in Definitions \ref{Def:fracderivRL} and \ref{Def:fracderivC}.
\end{theorem}

\begin{proof}
We use the following well-known relationship between the original Riemann--Liouville fractional differintegral and the standard repeated derivative:
\begin{equation}
\label{RLnD:relationship}
\prescript{RL}{c}I^{\alpha}_x\left(\frac{\mathrm{d}^n}{\mathrm{d}x^n}f(x)\right)=\prescript{RL}{c}I^{\alpha-n}_xf(x)-\sum_{j=0}^{n-1}\frac{(x-c)^{\alpha-n+j}}{\Gamma(\alpha-n+j+1)}f^{(j)}(c),\quad\quad\alpha\in\mathbb{C},n\in\mathbb{N}.
\end{equation}
Now recall the definitions \eqref{fracderivC} and \eqref{fracderivRL} for the new fractional derivative operators, and use the relation \eqref{RLnD:relationship} in each term of the series formula:
\begin{align*}
\left( {}_c\mathfrak{C}_{\alpha,\beta,\gamma}^{\delta;\omega_1,\omega_2}f\right) (x)&={}_c\mathfrak{I}_{\alpha,\beta,n-\gamma}^{-\delta;\omega_1,\omega_2}\left(\frac{\mathrm{d}^n}{\mathrm{d}x^n}f(x)\right) \\
&=\sum_{k=0}^{\infty }\sum_{l=0}^{\infty }\frac{(-\delta )_{k+l}\omega_1^k\omega_2^l}{k!l!}\prescript{RL}{c}I_x^{\alpha k+\beta l+(n-\gamma)}\left(\frac{\mathrm{d}^n}{\mathrm{d}x^n}f(x)\right) \\
&=\sum_{k=0}^{\infty }\sum_{l=0}^{\infty }\frac{(-\delta )_{k+l}\omega_1^k\omega_2^l}{k!l!}\Bigg[\prescript{RL}{c}I^{(\alpha k+\beta l+n-\gamma)-n}_xf(x) \\
&\hspace{4cm}-\sum_{j=0}^{n-1}\frac{(x-c)^{(\alpha k+\beta l+n-\gamma)-n+j}}{\Gamma((\alpha k+\beta l+n-\gamma)-n+j+1)}f^{(j)}(c)\Bigg] \\
&=\sum_{k=0}^{\infty }\sum_{l=0}^{\infty }\frac{(-\delta )_{k+l}\omega_1^k\omega_2^l}{k!l!}\prescript{RL}{c}I^{\alpha k+\beta l-\gamma}_xf(x) \\
&\hspace{2cm}-\sum_{k=0}^{\infty }\sum_{l=0}^{\infty }\sum_{j=0}^{n-1}\frac{(-\delta )_{k+l}\omega_1^k\omega_2^l}{k!l!}\cdot\frac{(x-c)^{\alpha k+\beta l-\gamma+j}}{\Gamma(\alpha k+\beta l-\gamma+j+1)}f^{(j)}(c) \\
&=\left( {}_c\mathfrak{D}_{\alpha,\beta,\gamma}^{\delta;\omega_1,\omega_2}f\right)(x)-\sum_{j=0}^{n-1}(x-c)^{-\gamma+j}E_{\alpha,\beta,-\gamma+j+1}^{-\delta}\Big(\omega_1(x-c)^{\alpha},\omega_2(x-c)^{\beta}\Big)f^{(j)}(c)
\end{align*}
\end{proof}


\begin{corollary}
Let $f\in C^n(c,d)$ be an $n$ times differentiable function such that $f^{(k)}(c)=0$ for all $k=0,1,2,\dots,n-1$. Then for any $\alpha,\beta,\gamma,\delta,\omega_1,\omega_2\in\mathbb{C}$ with $\mathrm{Re}(\alpha)>0$, $\mathrm{Re}(\beta)>0$, and $0<\mathrm{Re}(\gamma)<n$, the fractional derivative operators \eqref{fracderivRL} and \eqref{fracderivC} of Riemann--Liouville and Caputo type are equal to each other and form a two-sided inverse to the fractional integral operator \eqref{4}.
\end{corollary}

\begin{proof}
The condition on $\gamma$ means that $n\geq\lfloor\mathrm{Re}(\gamma)\rfloor+1$, so for sure all the initial value terms on the right-hand side of \eqref{relation:RLC} must be zero. This means
\[
\left( {}_c\mathfrak{C}_{\alpha,\beta,\gamma}^{\delta;\omega_1,\omega_2}f\right) (x)=\left( {}_c\mathfrak{D}_{\alpha,\beta,\gamma}^{\delta;\omega_1,\omega_2}f\right)(x).
\]
Similarly, all the initial value terms on the right-hand side of \eqref{inversion:IC} must be zero. So by \eqref{inversion:RLI} and \eqref{inversion:IC}, the fractional derivative operator (both $\mathfrak{D}$ and $\mathfrak{C}$ now being the same) is both a left inverse and a right inverse to the fractional integral operator.
\end{proof}

\section{Discussion and further work} \label{Sec:end}

In this paper, we have introduced a new Mittag-Leffler type function with two variables $x,y$ and four parameters $\alpha,\beta,\gamma,\delta$, which we called $E_{\alpha,\beta,\gamma}^{\delta}(x,y)$ and defined using a double infinite series of powers of $x$ and $y$. We have also given particular consideration to the special case where $x=t^{\alpha}$ and $y=t^{\beta}$ for a single variable $t$, giving rise to the function $t^{\gamma-1}E_{\alpha,\beta,\gamma}^{\delta}(t^{\alpha},t^{\beta})$ which is defined using a double infinite series of fractional powers of $t$. The motivation for this work comes from a number of different directions, which serve to prove the significance of this specific function.

\begin{itemize}
\item Firstly, this may be seen as a continuation of previous work on bivariate Mittag-Leffler functions, which has been ongoing in the last few years as a \textbf{pure mathematical} project. In particular, the second and third authors previously defined a bivariate Mittag-Leffler function in \cite{ozarslan-kurt}, and analysis of that function continued in \cite{kurt-ozarslan-fernandez}. In the current paper, we are continuing to develop the broader topic of bivariate Mittag-Leffler functions. \\

\item Secondly, the main motivation for our choice of this particular Mittag-Leffler function arises from the first author's correspondence with a team of researchers working in \textbf{bioengineering}. Specifically, they are applying a novel rheological approach to study cells and cellular materials, describing stress-strain relations by using fractional calculus to capture the viscoelastic behaviour of these materials. Their study of the mechanical properties of soft tissues is motivated by the aim of understanding disease developments. Some of their work has already appeared in \cite{bonfanti-etal}, and this already involves the classical Mittag-Leffler function $E_{\alpha,\beta}(x)$ of one variable and two parameters. In unpublished work shared with the first author, they have derived stress-strain relations in the more difficult case of three independent fractional springpots. It turns out that these relations can be described by an integral transform involving our bivariate Mittag-Leffler function in the kernel. This realisation -- that real biological processes can be modelled using a new bivariate Mittag-Leffler function -- was the main impetus for our work here. \\

It should be noted that the function emerging from these applications is of the form $t^{\gamma-1}E_{\alpha,\beta,\gamma}^{1}(t^{\alpha},t^{\beta})$, namely the univariate form with $\delta=1$. In this function there is no $\delta$ and no Pochhammer symbol, only a $(k+l)!$ on the numerator instead of $(\delta)_{k+l}$. Our reason for introducing the parameter $\delta$, inspired by the Pochhammer symbol in the numerator of the Prabhakar (3-parameter) Mittag-Leffler function $E_{\alpha,\beta}^{\rho}(x)$, is that including this Pochhammer symbol allows a semigroup property to be established (Theorem \ref{Thm:semigroup} above). If we eliminated $\delta$ by assuming $\delta=1$ always, then the composition of the integral operator with itself would be a new mystery operator; by allowing $\delta$ to vary, we can see that mystery operator as simply the same thing with $\delta=2$ instead of $\delta=1$. \\

\item Thirdly, the specific Mittag-Leffler function proposed here arises naturally from certain important \textbf{fractional differential equations}. We have already seen (Theorem \ref{Thm:FDE} and Corollary \ref{Coroll:FDE} above) how our Mittag-Leffler function appears as a solution of some very simple differential equations with two independent fractional orders of differentiation $\alpha,\beta$, in the same way as the original Mittag-Leffler function $E_{\alpha}(x)$ appears when there is only one fractional order of differentiation $\alpha$. The equations we found above can be studied further, and solved numerically to obtain approximations of our Mittag-Leffler function. \\

Furthermore, certain elementary systems of fractional differential equations also give rise to the same bivariate Mittag-Leffler function. Systems of fractional differential equations have so far been studied mostly from the numerical point of view, and applications have been discovered e.g. for electrical circuits \cite{kaczorek}. We are in contact with a research group working on systems of the form
\[
\begin{cases}
\prescript{C}{0}D^{\alpha}_tu(t)=Pu(t)+Qv(t), \\
\prescript{C}{0}D^{\beta}_tv(t)=Ru(t)+Sv(t),
\end{cases}
\]
where $\alpha,\beta$ are independent fractional orders of differentiation and $P,Q,R,S$ are arbitrary constants. We have discovered a connection between their (still unpublished) analytic solution and our bivariate Mittag-Leffler function. Since this system of fractional differential equations is very elementary, it is surprising that so far a full analytic solution does not exist in the literature. Again, it seems that our function has emerging importance in the field of fractional differential equations.
\end{itemize}

To summarise, the functions and operators defined in this paper find themselves at the intersection point of many important strands of research. The results we have proved concerning them are expected to be used immediately in some following works.

Future work using these Mittag-Leffler functions and associated differintegral operators is expected to include: numerical approximation of the functions by solving fractional differential equations; numerical approximation of the operators by Bernstein-polynomial techniques; asymptotic analysis of the functions; realisation of the operators in bioengineering applications; realisation of the functions in fractional differential equation systems; and many more properties waiting to be proved mathematically and then applied in practice.

\section*{Acknowledgements}
The authors would like to thank Alessandra Bonfanti and Louis Kaplan for stimulating discussions about interdisciplinary applications.

\end{document}